\documentclass[oneside,english]{amsart}
\usepackage[T1]{fontenc}
\usepackage[latin9]{inputenc}
\usepackage{amsthm}
\usepackage{amssymb}
\usepackage{esint}
\usepackage{color}

\makeatletter
\numberwithin{equation}{section}
\numberwithin{figure}{section}
\theoremstyle{plain}
\newtheorem{thm}{\protect\theoremname}
  \theoremstyle{plain}
  \newtheorem{cor}[thm]{\protect\corollaryname}
  \theoremstyle{remark}
  \newtheorem{rem}[thm]{\protect\remarkname}
  \theoremstyle{plain}
  \newtheorem{lem}[thm]{\protect\lemmaname}
  \theoremstyle{plain}
  \newtheorem{prop}[thm]{\protect\propositionname}
 \theoremstyle{definition}
 \newtheorem*{defn*}{\protect\definitionname}
  \theoremstyle{definition}
  \newtheorem{example}[thm]{\protect\examplename}
  \theoremstyle{definition}
  \newtheorem{defn}[thm]{\protect\definitionname}
 \theoremstyle{plain}
 \newtheorem{conjecture}[thm]{Conjecture}

\newtheorem{assm}[thm]{Assumption}

\makeatother

\usepackage{babel}
  \providecommand{\corollaryname}{Corollary}
  \providecommand{\definitionname}{Definition}
  \providecommand{\examplename}{Example}
  \providecommand{\lemmaname}{Lemma}
  \providecommand{\propositionname}{Proposition}
  \providecommand{\remarkname}{Remark}
\providecommand{\theoremname}{Theorem}

\begin{document}

\title{Freezing and decorated Poisson point processes}

\author{Eliran Subag}
\author{Ofer Zeitouni}\thanks{Research partially supported by a grant of the Israel Science Foundation.}

\begin{abstract}
The limiting extremal processes of the branching Brownian motion (BBM),
the two-speed BBM, and the branching random walk are known to be randomly
shifted decorated Poisson point processes (SDPPP). In the proofs of
those results, the Laplace functional of the limiting extremal process
is shown to satisfy $L\left[\theta_{y}f\right]=g\left(y-\tau_{f}\right)$
for any nonzero, nonnegative, compactly supported, continuous function
$f$, where $\theta_{y}$ is the shift operator, $\tau_{f}$ is a
real number that depends on $f$, and $g$ is a real function that
is independent of $f$. We show that, under some assumptions, this
property characterizes the structure of SDPPP. Moreover, when it holds,
we show that $g$ has to be a convolution of the Gumbel distribution
with some measure.

The above property of the Laplace functional is closely related to
a `freezing phenomenon' that is expected by physicists to occur in
a wide class of log-correlated fields, and which has played an important
role in the analysis of various models. Our results shed light on
this intriguing phenomenon and provide a natural tool for proving an
SDPPP structure in these and other models.
\end{abstract}
\maketitle

\section{Introduction}

The branching Brownian motion (BBM) is a continuous-time branching
process described as follows. At time $t=0$ a single particle starts
a standard Brownian motion $x$ from the origin, continuing for a
randomly distributed exponential time $T$ independent of $x$. At
this moment, the particle splits into to two particles. Each, in turn,
performs a Brownian motion starting from $x(T)$ and is subject to
the same splitting rule. Thus, at time $t$ there is a random number
of particles $N\left(t\right)$ and we denote their positions by $X_{1}\left(t\right),\ldots,X_{N\left(t\right)}\left(t\right)$.

The BBM has been extensively studied over the last decades. The seminal
works of Mckean \cite{McKean}, Bramson \cite{Bramson78,Bramson83},
and Lalley and Sellke \cite{Lalley&Sellke} were mainly concerned
with the maximum (or rightmost particle) $\mathcal{M}_{t}=\max_{i\leq N\left(t\right)}X_{i}\left(t\right)$
and its relations to the Fisher-Kolmogorov-Petrovsky-Piscounov (F-KPP)
equation \cite{KPP}. In particular, it was shown in \cite{Lalley&Sellke}
that with appropriate recentering term $m_{t}$,
\begin{equation}
\lim_{t\to\infty}\mathbb{P}\left\{ \mathcal{M}_{t}-m_{t}\leq x\right\} =\mathbb{E}\exp\left\{ -e^{-\sqrt{2}\left(x-\log\left(CZ\right)/\sqrt{2}\right)}\right\} ,\label{eq:maxBBM}
\end{equation}
where $Z$ is the limit of the so-called derivative martingale and
$C$ is a constant.

Recently, BBM became the object of renewed interest with the main
focus being the behavior of extreme values of the process \cite{ABBS,ABKGene,ABKPois,ABKergodic,ABKextremal,BruDeBBM}.
Perhaps the most important result in this direction is a remarkable
description of the limiting extremal process, i.e. the limit in distribution
\begin{equation}
  \label{eq-120314a}
\xi=\lim_{t\rightarrow\infty}\xi_{t}\triangleq\lim_{t\rightarrow\infty}\sum_{i\leq N\left(t\right)}\delta_{X_{i}\left(t\right)-m_{t}},
\end{equation}
given independently by Arguin, Bovier and Kistler \cite{ABKextremal}
and A\"{i}d\'{e}kon, Berestycki, Brunet and Shi \cite{ABBS}.

In the sequel, we denote by
$\overset{d}{=}$ equality in distribution.
For a point process $D=\sum_{i\geq 1} \delta_{d_i}$,
we denote by $\theta_{x} D$  the
shift of $D$ by
$x$,
i.e. $\theta_{x}D=\sum_{i\geq1}\delta_{d_{i}+x}$.

The following notions describe the structure of limits alluded to above.
\begin{defn}
(1) A point process $\psi$ is a decorated Poisson point process
(DPPP) of intensity $\nu$ and decoration $D$
(denoted $\psi\sim DPPP\left(\nu,D\right)$),
if $\psi\overset{d}{=}\sum_{i\geq1}\theta_{\zeta_{i}}D_{i}$
where $\zeta=\sum_{i\geq1}\delta_{\zeta_{i}}$ is a Poisson process
with intensity $\nu$, $D$ is some point
process, and $D_{i}$, $i\geq1$, are copies of $D$, independent of each
other and of $\zeta$.

(2) A point process
$\varphi$ is a randomly
shifted decorated Poisson point process (SDPPP)
of intensity $\nu$, decoration $D$ and shift $S$ (denoted
$\varphi\sim SDPPP\left(\nu,D,S\right)$) if for
$\psi\sim DPPP\left(\nu,D\right)$ and some independent
(of $\psi$)
random variable $S$ it holds that
$\varphi\overset{d}{=}\theta_{S}\psi$.
\end{defn}
In this notation \cite{ABBS,ABKextremal} showed that, with some point
process $D$, the limiting extremal process $\xi$ in
\eqref{eq-120314a}
satisfies
\[
\xi\sim SDPPP\left(e^{-\sqrt{2}x}dx,D,\log\left(\sqrt{2}CZ\right)/\sqrt{2}\right),
\]
where $C$, $Z$ are as in (\ref{eq:maxBBM}).

Our work centers around the relations between two properties of the
limiting extremal process, one being the specific structure we have
just described. The other is related to an intriguing  `freezing' phenomenon
observed first by Spohn and Derrida in \cite{DerSpo}, exploiting
Bramson's results on the F-KPP equation \cite{Bramson83}. They define
the function $G_{t,\beta}\left(y\right)\triangleq\mathbb{E}\exp\left\{ -e^{-\beta y}\sum_{i=1}^{N\left(t\right)}e^{\beta X_{i}\left(t\right)}\right\} $
and conclude it exhibits the shape of a `traveling wave' as $t\to\infty$.
That is, for some $m_{t,\beta}$ increasing in $t$,
\begin{equation}
\label{eq:frzseq1}
\lim_{t\to\infty}G_{t,\beta}\left(y+m_{t,\beta}\right)=g_{\beta}\left(y\right).
\end{equation}
Moreover, they show that the profile of the wave and the velocity
`freeze' at a certain transition temperature $\beta=\beta_{c}$ ($\beta$
is the inverse temperature):
\begin{equation}
\label{eq:frzseq2}
\mbox{for any }\beta>\beta_{c}:\,\, g_{\beta}\left(x\right)=g_{\beta_{c}}\left(x\right)\,\,\mbox{and}\,\, m_{t,\beta}+c_{\beta}=m_{t},
\end{equation}
with some constants $c_{\beta}$ depending on $\beta$.

We introduce the shift-Laplace functional of a point process $\xi$:
\[
L_{\xi}\left[\left.f\,\right|\, y\right]\triangleq\mathbb{E}\left\{ \exp\left(-\int\theta_{y}fd\xi\right)\right\} ,
\]
where $f:\mathbb{R}\rightarrow\mathbb{R}$ is measurable, nonnegative
function and where we abuse notation by writing $\theta_{y}f\left(x\right)=f\left(x-y\right)$.
Denote $f\approx g$ whenever two functions are equal up to translation
and let $\left[g\right]$ denote the equivalence class of $g$ under
this relation. Since $\xi_{t}\rightarrow\xi$ in distribution, it
is easily seen (under some boundedness condition, cf. Lemma \ref{lem:str_frz}) that freezing, put in other words, means that
\begin{equation}
\label{eq:frzphsc}
\mbox{for any }f\in\left\{ e^{\beta y}:\,\beta>\beta_{c}\right\}:\,\, L_{\xi}\left[\left.f\,\right|\,\cdot\,\right]\approx g_{\beta_{c}}\left(\cdot\right).
\end{equation}
Inspired by this, we introduce the following notion, using the notation
$C_{c}^{+}\left(\mathbb{R}\right)$ for the class of nonnegative,
compactly supported, continuous, real functions that are
not identically
equal to $0$.
\begin{defn}
  A shift-Laplace functional  is
  uniquely supported on $[g]$ if
$L_{\xi}\left[\left.f\,\right|\,\cdot\,\right]\approx g\left(\cdot\right)$
for any $f\in C_{c}^{+}\left(\mathbb{R}\right)$.
\end{defn}

The following is a direct corollary of our main result Theorem \ref{thm:main}.
As we shall see in Lemma \ref{lem:g monotone}, when the shift-Laplace
functional is uniquely supported the corresponding function $g$ is
monotone. Thus in order to simplify the notation we restrict below
to the case where $g$ is increasing. The point process $D\left(\xi\right)$
appearing in the statement is defined as the limit, as $y\to\infty$, of $\theta_{-y}\xi$ conditioned on
$\xi((y,\infty))>0$, shifted so its maximum is at $0$ (see Section \ref{sec:(Re)construction}).
Denote by $\mbox{Gum}\left(y\right)\triangleq\exp\left\{ -e^{-y}\right\} $
the standard Gumbel distribution function.
\begin{cor}
\label{cor:intro}Let \textup{$\xi$} be a point process such that
$\mathbb{P}\left\{ \xi\left(\mathbb{R}\right)>0\right\} =1$ and let
$g:\mathbb{\mathbb{R}}\rightarrow\mathbb{\mathbb{R}}$ be an increasing
function. Consider the following statements:
\begin{enumerate}
\item [(US)] $L_{\xi}\left[\left.f\,\right|\,\cdot\,\right]$ is uniquely
supported on $\left[g\right]$.
\item [(SUS)] $L_{\xi}\left[\left.f\,\right|\,\cdot\,\right]$ is uniquely
supported on $\left[g\right]$ and for some random variable $Z$,
\begin{equation}
g\left(y\right)=\int{\rm {Gum}}\left(c\left(y-z\right)\right)d\mu_{Z}\left(z\right)=\mathbb{E}\left\{ {\rm {Gum}}\left(c\left(y-Z\right)\right)\right\} ,\label{eq:g_form}
\end{equation}
where $\mu_{Z}$ is the law of $Z$ and $c>0$.
\item [(SDP)] $\xi\sim SDPPP\left(e^{-cx}dx,D,Z\right)$ for some point
process $D$, random variable $Z$,  and constant $c>0$.
\end{enumerate}
If $g$ satisfies (\ref{eq:21-1}) below, then (SUS) and (SDP) are
equivalent and (SDP) holds with $D$ equal to $D\left(\xi\right)$
up to translation. If, in addition, $D\left(\xi\right)$ satisfies
a certain boundedness condition (Assumption \ref{Assumption:decoration}),
or if $L_{\xi}\left[\left.e^{\beta x}\,\right|\,\cdot\,\right]\approx g\left(\cdot\right)$ for all $\beta>c$,
then all three conditions are equivalent.
\end{cor}
By Corollary \ref{cor:frz_specific_func} below, in the presence of (US),
the condition $L_{\xi}\left[\left.e^{\beta x}\,\right|\,\cdot\,\right]\approx g\left(\cdot\right)$
is equivalent to the condition
$\mathbb P\left\{\int e^{\beta x} d\xi (x) < \infty \right\}=1$.
We conjecture (see Conjecture \ref{conj1} below), that the assumptions in the last sentence
in Corollary \ref{cor:intro} are not needed for
 the conclusion.

In the case of BBM, where all three conditions are already known to
occur, the implications of the result are rather conceptual than practical.
However, in other models it  allows one to prove that the limiting
extremal process admits the representation of (SDP) by studying its
Laplace functional. In fact, as we discuss in Section \ref{sec:Relations},
a study of the Laplace functional is the main step in the approach
taken for all the models for which the limiting extremal process is
known to satisfy (SDP): the BBM \cite{ABKextremal}, the two-speed
BBM \cite{BHartung}, and the branching random walk (BRW) \cite{MadauleBRWext}.
In particular, Proposition 3.2 of \cite{ABKextremal} gives exactly
Condition (SUS) above.

We mention the class of logarithmically-correlated (log-correlated,
for short) Gaussian fields \cite{A&Z,CLD,DRSV2012,F&B1,MadauleMax},
which are of great importance in a variety of fields (see Section
\ref{sec:Relations}). Log-correlated Gaussian fields are known to
share many properties with BBM. For example, the distribution of the
maximum \cite{Louidor,BDZGff,MadauleMax}, the overlap of extremal
points \cite{ABKGene,ABKPois,A&Z}, and their limiting Gibbs measures
\cite{RVreview,Webb} behave similarly. It is therefore widely believed
that their limiting extremal processes exist and should also exhibit
the structure of Condition (SDP). Interestingly, Carpentier and Le
Doussal \cite{CLD} postulate that the freezing phenomenon described
above is also satisfied by a wide class of log-correlated Gaussian
fields. For the sub-class of star scale invariant fields freezing
was proved in \cite{MRVfrz}. Other works in the physics literature
that are related to freezing include \cite{Fyodorov,F&B1,Fyo3,Fyo2,Fyo1}.
From those studies it is evident that the freezing phenomenon plays
an important role in the analysis of log-correlated fields. We hope
the link we have made between freezing and the structure of the extremal
process will improve the understanding of both phenomena. In light
of the above we also believe that it provides a natural tool for attacking
the problem of proving (SDP) for the limiting extremal process of
log-correlated fields.

Lastly, let us briefly mention the related problem of the characterization
of decorated Poisson point processes (with no shift, $Z=0$). A point
process is said to be exponentially-$c$-stable, $c>0$, if for three
independent copies of it $\xi$, $\xi_{1}$ and $\xi_{2}$ and any
two numbers $a$, $b$ such that $e^{a}+e^{b}=c$, $\xi\overset{d}{=}\theta_{a}\xi_{1}+\theta_{b}\xi_{2}$.
Brunet and Derrida \cite{BruDeBBMPreliminary} (p. 18) conjectured
that this property is equivalent to (SDP) with $Z=0$. This can be proved
using a representation of the Laplace functional for infinitely divisible processes (see Maillard \cite{Pascal}; implicitely, this also appears in
\cite{DMZ}). It is fairly
simple to see that exponential-$c$-stability is equivalent to uniqueness
of the support of $L_{\xi}\left[\left.f\,\right|\, x\right]$ up to
translation with $g\left(x\right)=\mbox{Gum}\left(cx\right)$ (cf.
Corollary \ref{cor:exp-stb}). Hence, our main result can be seen
as a generalization of this characterization. Let us stress, however,
that the assumption of exponential-stability constraints $g$ to be
of the form just specified and so in this case the relation to Condition
(US) is irrelevant. Apart from the equivalence of exponential-stability
and (SDP) with $Z=0$, a corollary of our main result yields a description
of the decoration process $D$ in terms of the original process $\xi$.

In the next section we give some further definitions and state our
main results. In Section \ref{sec:Relations} we discuss in more detail
the relations to the works mentioned above from the mathematics and
physics literature. We prove one (easy) direction of our main theorem
in Section \ref{sec:dppp_to_frz}. Sections \ref{sec:Basic_properties}
and \ref{sec:(Re)construction} are devoted to results we shall need
in the proof of the other direction, given in Section \ref{sec:mainpf}.
In particular, in Section \ref{sec:(Re)construction} we shall describe
a construction of a point process which will be used as the decoration
of a specific SDPPP process we use in our proofs.
In Section
\ref{sec:Relation-to-Freezing} we discuss the relations between freezing for a limiting process \eqref{eq:frzphsc}
and the definition in terms of the sequence of processes \eqref{eq:frzseq1}, \eqref{eq:frzseq2}.
Finally, a short appendix is devoted to the proof
of a relation between Gumbel distribution functions of different scales.

\section{\label{sec:Main-Results}Main Results}

Some of the results rely on the following assumption. As Proposition
\ref{prop:nice xi} of Section \ref{sec:Basic_properties} states,
equation (\ref{eq:21-1}) holds whenever the intensity measure of
the process is finite on some nonempty, open set.

\begin{assm}

\label{Assumption:regularity-1}The shift-Laplace functional $L_{\xi}\left[\left.f\,\right|\,\cdot\,\right]$
is uniquely supported on $\left[g\right]$ where $g$ is increasing
and $g$ satisfies, with some constant $c=c_{\xi}>0$,
\begin{equation}
\lim_{x\rightarrow\infty}\frac{1-g\left(x+y\right)}{1-g\left(x\right)}=\lim_{x\rightarrow\infty}\frac{\log g\left(x+y\right)}{\log g\left(x\right)}=e^{-cy}.\label{eq:21-1}
\end{equation}
\end{assm}

In Section \ref{sec:(Re)construction} we shall construct, under a
weaker assumption than the above, the point process $D\left(\xi\right)$
(see Definition PP-D). For some of our results we will require it to have one of the
following properties.
\begin{defn}
  (1) The  point process
$D(\xi)$ is said to have
exponential moments if there exist some
$t,\,\epsilon>0$ such that $\mathbb{E}\left[\exp\left\{ tD\left(\xi\right)\left(-\epsilon,0\right)\right\} \right]<\infty$.\\
(2)
The point process $D(\xi)$ is said to satisfy a law of large numbers (LLN) with rate $c$ if
for some function $\alpha:\,\mathbb{R}\rightarrow\mathbb{R}$
and some constant $u>0$,
the point process $\psi_c\sim DPPP\left(e^{-cx}dx,D(\xi)\right)$ satisfies
\begin{equation}
\label{eq:alphalim}
\frac{\psi_c\left(\left(-y,\infty\right)\right)}{\alpha\left(y\right)}\overset{prob.}{\longrightarrow}u,\,\,\,\mbox{as }y\rightarrow\infty.
\end{equation}
\end{defn}
\begin{rem}
By similar arguments to those in Section
\ref{sec:dppp_to_frz}, it can be shown  that
\begin{equation}
\label{eq:llnsum}
\psi_c\left(\left(-y,\infty\right)\right)\overset{d}{=}\sum_{i=1}^{N_{y}}\theta_{X_{i}}D_{i}\left((0,\infty)\right),
\end{equation}
where $N_{y}\sim\mbox{Pois}\left(e^{cy}/c\right)$, $X_{i}\overset{d}{=}X\sim\exp\left(c\right)$,
and $D_{i}\overset{d}{=}D(\xi)$ are all independent. Thus, if $\mathbb E \theta_X D(\xi)((0,\infty))<\infty$,  it follows by the law of large numbers that
$D(\xi)$ satisfies an LLN with rate $c$. That is,
\[
\frac{\psi_c\left(\left(-y,\infty\right)\right)}{e^{cy}/c}\overset{prob.}{\longrightarrow}\mathbb E \theta_X D(\xi)((0,\infty)),\,\,\,\mbox{as }y\rightarrow\infty\,.
\]
\end{rem}
\begin{rem}
From \eqref{eq:llnsum} it follows that for any natural $k$,
\[
\psi_c\left(\left(-(y+\log k/c),\infty\right)\right)\overset{d}{=}\sum_{j=1}^{k}\psi_c^{(j)}\left(\left(-y,\infty\right)\right)\, ,
\]
where $\psi_c^{(j)}$ are i.i.d copies of $\psi_c$.
Assuming \eqref{eq:alphalim}, it is therefore straightforward to verify that for any rational $q$,
\[
\lim_{y\to\infty}\frac{\psi_c\left(\left(-(y+\log q/c),\infty\right)\right)}{\alpha(y+\log q/c)}=q\lim_{y\to\infty}\frac{\alpha(y)}{\alpha(y+\log q/c)}\frac{\psi_c\left(\left(-y,\infty\right)\right)}{\alpha(y)}\, ,
\]
and thus, since $\psi_c\left(\left(-y,\infty\right)\right)$ is increasing in $y$, for any real $r$,
\begin{equation}
\label{eq:alpharegvar}
\lim_{y\to\infty}\frac{\alpha(y+t)}{\alpha(y)}=e^{ct}.
\end{equation}
\end{rem}

\begin{assm}

\label{Assumption:decoration}Let $\xi$ be a point process that satisfies
Assumption \ref{Assumption:regularity-1} with constant $c_{\xi}$.
Assume that the decoration process $D\left(\xi\right)$ has exponential
moments, or that $D\left(\xi\right)$
satisfies an LLN with rate $c_{\xi}$.\end{assm}

The following is our main result which, under the assumptions above,
shows that uniqueness of the support of the shift-Laplace functional
- a property of the Laplace functional of the process - characterizes
the structure of randomly shifted decorated Poisson point process.
The conditions (US), (SUS), and (SDP) considered here are the ones
stated in Corollary \ref{cor:intro}. Whenever $L_{\xi}\left[\left.f\,\right|\,\cdot\,\right]$
is uniquely supported on $\left[g\right]$, for any $f\in C_{c}^{+}\left(\mathbb{R}\right)$,
we define $\tau_{f}=\tau_{f}^{g}\left(\xi\right)$ to be the (unique,
whenever $\mathbb{P}\left\{ \xi\left(\mathbb{R}\right)>0\right\} >0$,
as follows from Lemma \ref{lem:g monotone}) translation such that
$L_{\xi}\left[\left.f\,\right|\, y\right]=g\left(y-\tau_{f}\right)$.
For a point process $\xi$ define the maximum (or rightmost particle)
as $\mathcal{M}\left(\xi\right)\triangleq\inf\left\{ y\in\mathbb{R}:\,\xi\left(\left(y,\infty\right)\right)=0\right\} $,
where if the set is empty we take the infimum to be $\infty$.
\begin{thm}
\label{thm:main}Let \textup{$\xi$} be a point process such that
$\mathbb{P}\left\{ \xi\left(\mathbb{R}\right)>0\right\} =1$ and let
$g:\mathbb{\mathbb{R}}\rightarrow\mathbb{\mathbb{R}}$ be an increasing
function. Then the following hold:
\begin{description}
\item [{Converse~part}] If (SDP) holds, then (SUS) holds with the same
$Z$ and $c$ and, in particular, (US) holds. Moreover, in this case
the corresponding shifts are given, for any $f\in C_{c}^{+}\left(\mathbb{R}\right)$,
by
\begin{equation}
\tau_{f}^{g}\left(\xi\right)=c^{-1}\log\left(-\int_{-\infty}^{\infty}e^{-ct}\left(L_{D}\left[\left.f\,\right|\,-t\right]-1\right)dt\right).\label{eq:31}
\end{equation}

\item [{Direct~part}] Let Assumption \ref{Assumption:regularity-1} hold
with $c_{\xi}$. \end{description}
\begin{enumerate}
\item [(a)] \label{thm:Main-a}If (SUS) holds with some random variable
$Z$ and with $c=c_{\xi}$, then (SDP) holds with the same $Z$ and
$c$ and with $D=\hat{D}\left(\xi\right)\triangleq\theta_{-\tau_{\mathcal{M}}^{g}\left(\xi\right)-c^{-1}\log c}D\left(\xi\right)$,
where $\tau_{\mathcal{M}}^{g}\left(\xi\right)$ is defined in Lemma
\ref{lem:maxFin} and $D\left(\xi\right)$ is the point
process satisfying $\mathcal{M}\left(D\left(\xi\right)\right)=0$
a.s., given in Definition PP-D of Section \ref{sec:(Re)construction}.
\item [(b)] \label{thm:Main-b}If (US) holds, then there exists a random
variable $Z$ such that $\xi\overset{d}{=}\theta_{Z}\psi$, where
$\psi\sim DPPP\left(e^{-c_{\xi}x}dx,\hat{D}\left(\xi\right)\right)$
possibly depends on $Z$.
\item [(c)] \label{thm:Main-c}
  In addition to Assumption \ref{Assumption:regularity-1}, suppose that
  $L_{\xi}\left[\left.e^{\beta x}\,\right|\,\cdot\,\right]\approx
  g\left(\cdot\right)$ for all $\beta>c=c_\xi$.
  If (US) holds then there exists a random variable $Z$ such that (SUS) holds with $c=c_{\xi}$.
\item [(d)] \label{thm:Main-d} In addition to Assumption \ref{Assumption:regularity-1}, suppose that Assumption \ref{Assumption:decoration} holds.
If (US) holds then there exists a random variable $Z$ such that (SUS)
holds with $c=c_{\xi}$.
\end{enumerate}
\end{thm}
We note again that, by Corollary \ref{cor:frz_specific_func},
under the assumption (US) the condition in part (c) is equivalent
to the condition
$\mathbb P\left\{\int e^{\beta x} d\xi (x) < \infty \right\}=1$
for all $\beta>c$.

We believe, but have been unable to prove, the following.
\begin{conjecture}
  \label{conj1}
The assumptions in parts (c) and (d) of
Theorem \ref{thm:main} are  not needed; that is, under Assumption
\ref{Assumption:regularity-1}, (US), (SUS) and (SDP) are equivalent.
\end{conjecture}
We remark that
Lemma \ref{lem:150414-1} of Section \ref{sec:mainpf} below implies that
freezing in the sense of \eqref{eq:frzphsc},
even without assuming (US), implies that $g$ is of the form \eqref{eq:g_form}.

The following corollary follows from the theorem. (However, we shall
prove it without relying on Theorem \ref{thm:main}, using a simpler
result in Section \ref{sec:(Re)construction}.)
\begin{cor}
\label{cor:exp-stb}Let $\xi$ be a point process such that $\mathbb P\left\{\xi(\mathbb R)>0\right\}>0$ and let $c>0$.
The following are equivalent:
\begin{enumerate}
\item \label{enu:exp-stb-1}$\xi$ is a DPPP with density $e^{-cx}dx$.
\item \label{enu:exp-stb-2}$\xi$ is $c$-exponentially-stable.
\item \label{enu:exp-stb-3}$L_{\xi}\left[\left.f\,\right|\,\cdot\,\right]$
is uniquely supported on the class $\left[{\rm {Gum}}\left(cy\right)\right]$.
\end{enumerate}

When any of the conditions hold, we can take the decoration to be
$D=\hat{D}\left(\xi\right)$, i.e., $\xi\sim DPPP\left(e^{-cx}dx,\hat{D}\left(\xi\right)\right)$,
where $\hat{D}\left(\xi\right)$ is defined in part (a) of Theorem
\ref{thm:main}.

\end{cor}

\section{\label{sec:Relations}Relations to other works}

\subsection{\label{sub:known_models}SDPPP in BBM, two-speed BBM and BRW}

In this section we discuss three processes for which the extremal
process, i.e. $\sum_{i\leq N\left(t\right)}\delta_{X_{i}\left(t\right)-m_{t}}$
with appropriate translations $m_{t}$, is known to converge to an
SDPPP of exponential density. The case of BBM was proved independently
by Arguin, Bovier and Kistler \cite{ABKextremal} and A\"{i}d\'{e}kon, Berestycki,
Brunet and Shi \cite{ABBS}, with each giving a different description
for the decoration process. The approach of A\"{i}d\'{e}kon et al. relies
on the so-called spinal decomposition - a tilted measure which distinguishes
the path of a single particle that typically attains extreme values
(i.e., the spine). They express the decoration as the limit, first
letting $t\to\infty$ and then letting $\zeta\to\infty$, of the point
measure of particles at time $t$ which have branched off the particle
at $\mathcal{M}_{t}$ after time $t-\zeta$, including the particle
at $\mathcal{M}_{t}$, all shifted by the position of $\mathcal{M}_{t}$.

The proof of Arguin et al. starts with a computation of the Laplace
functional of the limiting extremal process, based on its relation
to the F-KPP equation. In our notation, they prove Condition (SUS)
and express the corresponding shifts $\tau_{f}$ in terms of a solution
of the F-KPP equation with initial condition $v\left(0,y\right)=\exp\left\{ -f\left(y\right)\right\} $.
They then show that the Laplace functional of the limiting extremal
process is equal to that of an auxiliary process they construct -
a limit of SDPPP processes of density $-\sqrt{2/\pi}xe^{-\sqrt{2}x}dx$
(each has a different decoration processes). This shows that the two
limiting processes are equal in distribution and allows them to study
the latter in order to prove the required structure.

With $\eta_{t}\triangleq\sum_{i\leq N\left(t\right)}\delta_{X_{i}\left(t\right)-\sqrt{2}t}$,
Arguin et al. express the decoration of the extremal process as the
limit of $\theta_{-\mathcal{M}\left(\eta_{t}\right)}\eta_{t}$ conditioned
on $\mathcal{M}\left(\eta_{t}\right)>0$ as $t\to\infty$ (Chauvin
and Rouault \cite{CR} studied the same process). Our description
of the decoration process appearing in Theorem \ref{thm:main} also
involves the behavior of the process around high levels conditioned
on the maximum being sufficiently high. However, our results apply
to the limiting process $\xi$ directly. We study the process $\theta_{-y}\xi\triangleq\lim_{t\to\infty}\sum_{i\leq N\left(t\right)}\delta_{X_{i}\left(t\right)-m_{t}-y}$
conditioned on $\mathcal{M}\left(\xi\right)>y$, as $y\to\infty$.
For the purpose of comparison, in the case of BBM we can relate our
results to studying the limiting behavior as $t\to\infty$ and then
$y\to\infty$, and the approach of Arguin et al. can be seen as taking
the limits simultaneously by defining $y\left(t\right)=\sqrt{2}t-m_{t}=(3/2\sqrt{2})\log t-c+o(1)$
and letting $t\to\infty$.

Madaule \cite{MadauleBRWext} proved that the extremal process of
the BRW is an SDPPP of exponential density. Theorem 2.3 of \cite{MadauleBRWext},
which as the author notes is the key step to the main result, expresses
the Laplace functional of the extremal process shifted by the derivative
martingale, the definition of which is similar to that in the case
of BBM, for functions of the form $f\left(x\right)=\sum_{i\leq k}\theta_{i}\exp(\beta_{i}x)$
with $\beta_{i}$ larger then a critical value. It also gives the
independence of the derivative martingale and the limiting shifted
process. Combined with Remark 3.2, Theorem 2.3 of \cite{MadauleBRWext}
implies that for those functions $L_{\xi}\left[\left.f\,\right|\,\cdot\,\right]\approx\mbox{Gum}$.
Essentially, using an approximation argument the equivalence is extended
to $f\in C_{c}^{+}\left(\mathbb{R}\right)$ in order to show that
this process is exponentially-stable (to be accurate, the approximation
is done in terms of the characteristic function and not the Laplace
transform). By the result of \cite{Pascal}, this yields the required
structure, without, however, saying anything about the decoration
process. Applying Corollary \ref{cor:exp-stb} yields a description
in terms of the limiting process.

Very recently the two-speed BBM was considered by Bovier and Hartung
in \cite{BHartung}. The two-speed BBM is a variant of the BBM where
instead of constant variance (per time unit), the Brownian motions
describing the evolution of the particles have a certain variance
$\sigma_{1}^{2}$ up to a fraction $bt$, $b\in\left[0,1\right]$,
of the total time and some other variance $\sigma_{2}^{2}$ for the
rest of the evolution. As was shown in \cite{BHartung}, the structure
of the extremal process depends on the relation between $\sigma_{1}$
and $\sigma_{2}$, and in both cases is an SDPPP of exponential density.
Their method of proof essentially follows that of \cite{ABKextremal}.

Lastly, we mention that partial results on the structure of the extremal
process of the 2-dimensional discrete Gaussian free field are proved
by Biskup and Louidor \cite{Louidor}.

\subsection{\label{sub:Freezing-and-Log-correlated}Freezing and Log-correlated
fields }

Log-correlated random fields are fields whose covariance function
decays logarithmically with the distance. They have been analyzed
by Carpentier and Le Doussal (C\&LD) \cite{CLD} in a general setting
in their study of random energy landscapes. Various specific physical
models of log-correlated fields have been considered in \cite{Fyodorov,F&B1,Fyo3,Fyo2,Fyo1}.
Log-correlated fields are also of great significance in the area of
Gaussian multiplicative chaos, introduced by Kahane \cite{Kahane},
which has recently became the object of renewed interest \cite{MulCha2,MulCha3,MulCha4,RVreview,MulCha1}.
We also mention the 2-dimensional Gaussian free field which plays
an important role in statistical physics, the theory of random surfaces,
and quantum field theory \cite{GFF1,GFF2,GFF3,SheffGFF}.

One of the main motivations of the current work is the conjectured,
and in some cases proven, freezing phenomenon in log-correlated Gaussian
fields. The analysis of C\&LD \cite{CLD}, albeit non-rigorous, suggests
that the freezing phenomenon occurs in a wide class log-correlated
Gaussian fields. For the sub-class of star scale invariant fields
freezing was proved by Madaule, Rhodes and Vargas in \cite{MRVfrz}.
Freezing is also proved in the case of the Gaussian BRW in the work
of Webb \cite{Webb}.

Discussions on the implications of freezing in different models can
be found in \cite{Fyodorov,F&B1,Fyo3,Fyo2,Fyo1}. Of particular importance
to us is the work of Fyodorov and Bouchaud (F\&B) \cite{F&B1}. Assuming
freezing as their starting point, F\&B analyze the distribution of
the maximum of a specific log-correlated Gaussian field (see also
\cite{Fyo3} where the connection of this model to characteristic
polynomials of the CUE matrix is discussed). This allows them to conjecture
the limiting distribution of the recentered maximum to be
\[
\lim_{t\to\infty}\mathbb{P}\left\{ \mathcal{M}_{t}-m_{t}\leq x\right\} =g\left(x\right)=2e^{-\beta_{c}x/2}K_{1}\left(2e^{-\beta_{c}x/2}\right),
\]
where $g$ is the corresponding function from freezing, $K_{1}$ is
the modified Bessel function, and $\beta_{c}$ is the inverse of the
critical temperature.

In view of Theorem \ref{thm:main} and
Lemma \ref{lem:150414-1}, one
must wonder if the function $g$ above is of the structure of (\ref{eq:g_form}).
Indeed, there is a result of Gumbel himself \cite{Gumbel} by which,
reassuringly,
\begin{equation}
2e^{-\beta_{c}x/2}K_{1}\left(2e^{-\beta_{c}x/2}\right)=\mbox{Gum}*\mbox{Gum}'\left(\beta_{c}x\right),\label{eq:gum*gum}
\end{equation}
where $\mbox{Gum}'$ is the derivative of the standard Gumbel distribution.
Curiously, this implies that
\[
\mathcal{M}_{t}-m_{t}\overset{d}{\rightarrow}X'+X'',\,\,\,\mbox{as }t\to\infty,
\]
where $X'$ and $X''$ are i.i.d variables with distribution function
$\mbox{Gum}\left(\beta_{c}x\right)$. We do not have a good direct
explanation or proof as to why the shift in the F\&B model is itself
Gumbel-distributed.

The last example naturally leads us to discuss the subject of `universality
classes'. In the physics literature the term universality class refers
to a class of models that share a certain property. In the context
of log-correlated fields, C\&LD and F\&B were interested in the universality
class of fields such that the limiting distribution function of the
recentered maximum has certain properties.

The first random energy model considered by physicists is a collection
of uncorrelated Gaussian random variables introduced by Derrida in
\cite{REM}. In this case, by classical results from extreme value
theory \cite{deHaan} the limiting distribution of the recentered
maximum is the Gumbel distribution. C\&LD and F\&B emphasized the
fact that for the models they considered the limiting distribution
is different from this one, and thus the models are not in the same
universality class. C\&LD come to this conclusion by observing that
the tails of the distributions are different.

However, as mentioned they do expect freezing, which also occurs in
the case of uncorrelated variables, to occur in those models. In fact,
freezing would follow, for example under the conditions of Lemma \ref{lem:str_frz},
if the limiting extremal process is an SDPPP. Moreover, on the heuristic
level Theorem \ref{thm:main} says that freezing `almost' implies
such structure. The SDPPP structure would also allow us to interpret
the difference in the limiting distribution simply as the difference
in the corresponding random shift (in particular, in agreement with
(\ref{eq:gum*gum})).

\section{\label{sec:dppp_to_frz}Proof of Theorem \ref{thm:main}: the converse
part}

While proving the direct part of Theorem \ref{thm:main} requires
the development of new tools, the proof of the converse direction
is immediate. Moreover, the two are essentially independent, and so,
we shall deal with the latter now.

Note first that the case where $\xi\sim SDPPP\left(e^{-cx}dx,D,Z\right)$
with general random shift $Z$ easily follows, by conditioning, from
the case where
\[
\xi\sim SDPPP\left(e^{-cx}dx,D,0\right)=DPPP\left(e^{-cx}dx,D\right).
\]
For simplicity we also assume that $c=1$, i.e., $\xi\sim DPPP\left(e^{-x}dx,D\right)$.
The case with general $c\neq0$ follows by scaling. The proof for
the case $\xi\sim DPPP\left(e^{-x}dx,D\right)$ in fact follows from
Theorem 3.1 of \cite{Pascal}. We repeat the proof (though with a
slightly more direct approach) for the sake of completeness.

Fix some $f\in C_{c}^{+}\left(\mathbb{R}\right)$ throughout the proof.
Let $\zeta_{i}$ be the atoms of the Poisson process with intensity
$e^{-x}dx$ corresponding to the DPPP and define, for any $T>0$,
the random variable $I\left(T\right)\triangleq\left\{ i:\,\zeta_{i}\geq-T\right\} $
and the point process $\xi_{T}\triangleq\sum_{i\in I\left(T\right)}\theta_{\zeta_{i}}D_{i}$.
By the monotone convergence theorem,
\begin{align*}
L_{\xi}\left[\left.f\,\right|\, y\right] & =\lim_{T\rightarrow\infty}\mathbb{E}\exp\left\{ -\left\langle \theta_{y}f,\xi_{T}\right\rangle \right\} .
\end{align*}

By definition, $I\left(T\right)$ is a Poisson random variable of
parameter $e^{T}=\int_{-T}^{\infty}e^{-x}dx$ independent of $D_{i}$,
$i\geq1$. Conditioned on the event $\left\{ I\left(T\right)=k\right\} $,
$\xi_{T}$ is the sum of $k$ i.i.d point processes, each has the
same law as $\theta_{X-T}D$, where $X\sim\exp\left(1\right)$ (i.e.,
$X-T$ has density $e^{-t}\cdot\mathbf{1}_{\left[-T,\infty\right)}\left(t\right)dt/e^{T}$,
where $\mathbf{1}_{A}$ denotes the indicator function of the set
$A$). Thus
\begin{align*}
L_{\xi}\left[\left.f\,\right|\, y\right] & =\lim_{T\rightarrow\infty}\mathbb{E}\left[\mathbb{E}\left[\left.\exp\left\{ -\left\langle \theta_{y}f,\xi_{T}\right\rangle \right\} \right|I\left(T\right)\right]\right]\\
 & =\lim_{T\rightarrow\infty}\mathbb{E}\left\{ \left(\mathbb{E}\left\{ L_{D}\left[\left.f\,\right|\, y+T-X\right]\right\} \right)^{I\left(T\right)}\right\} .
\end{align*}

Recall that if $N\sim\mbox{Pois}\left(\lambda\right)$, then $\mathbb{E}t^{N}=\sum_{k=0}^{\infty}t^{k}\lambda^{k}e^{-\lambda}/k!=e^{\lambda\left(t-1\right)}$
for any $t\in\mathbb{R}$. Hence, by the monotone convergence theorem,
\begin{align*}
L_{\xi}\left[\left.f\,\right|\, y\right] & =\lim_{T\rightarrow\infty}\exp\left\{ e^{T}\left(\mathbb{E}\left\{ L_{D}\left[\left.f\,\right|\, y+T-X\right]\right\} -1\right)\right\} \\
 & =\lim_{T\rightarrow\infty}\exp\left\{ \int_{-T}^{\infty}e^{-t}\left(L_{D}\left[\left.f\,\right|\, y-t\right]-1\right)dt\right\} \\
 & =\exp\left\{ \int_{-\infty}^{\infty}e^{-t}\left(L_{D}\left[\left.f\,\right|\, y-t\right]-1\right)dt\right\} \\
 & =\exp\left\{ e^{-y}\int_{-\infty}^{\infty}e^{-t}\left(L_{D}\left[\left.f\,\right|\,-t\right]-1\right)dt\right\} ,
\end{align*}
and therefore, with $\tau_{f}^{g}=\log\left(-\int_{-\infty}^{\infty}e^{-t}\left(L_{D}\left[\left.f\,\right|\,-t\right]-1\right)dt\right)$,
\[
L_{\xi}\left[\left.f\,\right|\, y\right]=\exp\left\{ -e^{-\left(y-\tau_{f}^{g}\right)}\right\} =\mbox{Gum}\left(y-\tau_{f}^{g}\right).
\]

Since $f\in C_{c}^{+}\left(\mathbb{R}\right)$ was arbitrary, $L_{\xi}\left[\left.f\,\right|\,\cdot\,\right]$
is uniquely supported on $\left[\mbox{Gum}\left(y\right)\right]$.\qed

\section{\label{sec:Basic_properties}Basic properties under (US) }

In this section we investigate basic properties of the function $g$
and the point process $\xi$, assuming $L_{\xi}\left[\left.f\,\right|\,\cdot\right]$
is uniquely supported on $\left[g\right]$. In particular, we shall
consider the implications of assuming the intensity measure of $\xi$,
i.e., the Borel measure $\nu_{\xi}\left(B\right)=\mathbb{E}\left\{ \xi\left(B\right)\right\} $,
is boundedly finite.
\begin{lem}
\label{lem:g monotone}  Let $\xi$ be a point process that is not identically $0$, i.e. $\mathbb{P}\left\{ \xi\left(\mathbb{R}\right)>0\right\} >0$.
If the shift-Laplace functional $L_{\xi}\left[\left.f\,\right|\,\cdot\right]$
 is uniquely supported on $\left[g\right]$, then $g$ is a continuous, monotone function
such that for any $x\in\mathbb{R}$,
\begin{equation}
\mathbb{P}\left\{ \xi\left(\mathbb{R}\right)=0\right\} =\inf_{y\in\mathbb{R}}g\left(y\right)<g\left(x\right)<\sup_{y\in\mathbb{R}}g\left(y\right)=1.\label{eq:gbounds}
\end{equation}
\end{lem}
\begin{proof}
Since
\begin{align*}
g\left(y-\tau_{f}\right) & =\mathbb{P}\left\{ \xi\left(\mathbb{R}\right)=0\right\} \\
 & +\mathbb{P}\left\{ \xi\left(\mathbb{R}\right)>0\right\} \mathbb{E}\left[\left.\exp\left\{ -\left\langle \theta_{y}f,\xi\right\rangle \right\} \right|\xi\left(\mathbb{R}\right)>0\right],
\end{align*}
it is enough to prove the lemma under the assumption that $\xi\left(\mathbb{R}\right)>0$
almost surely. Assume henceforth it is so.

Fix some $f\in C_{c}^{+}\left(\mathbb{R}\right)$ throughout the proof.
By the bounded convergence theorem, $\lim_{y\rightarrow y_{0}}\left\langle \theta_{y}f,\xi\right\rangle =\left\langle \theta_{y_{0}}f,\xi\right\rangle $
almost surely and $\lim_{y\rightarrow y_{0}}L_{\xi}\left[\left.f\,\right|\, y\right]=L_{\xi}\left[\left.f\,\right|\, y_{0}\right]$,
i.e. $g$ is continuous.

Applying, again, the bounded convergence theorem to the sequence $\left\{ n^{-1}f\right\} $,
we obtain from uniqueness of the support,
\[
1\geq\sup_{y\in\mathbb{R}}g\left(y\right)\geq\lim_{n\rightarrow\infty}\mathbb{E}\left\{ \exp\left(-\left\langle n^{-1}f,\xi\right\rangle \right)\right\} =1.
\]

Similarly, by considering a sequence $f_{n}\in C_{c}^{+}\left(\mathbb{R}\right)$
such that $f_{n}\geq n\cdot\mathbf{1}_{\left(-n,n\right)}$ pointwise
and noting that $\left\langle f_{n},\xi\right\rangle \rightarrow\infty\cdot\mathbf{1}_{\left\{ \xi\left(\mathbb{R}\right)>0\right\} }$
almost surely, we conclude that $\inf_{y\in\mathbb{R}}g\left(y\right)=\mathbb{P}\left\{ \xi\left(\mathbb{R}\right)=0\right\} =0$.

If there were a $y_{0}$ such that $g\left(y_{0}\right)=\inf_{y\in\mathbb{R}}g\left(y\right)$,
then for some $y_{f}$, $L_{\xi}\left[\left.f\,\right|\, y_{f}\right]=0$,
which would imply that $\left\langle \theta_{y_{f}}f,\xi\right\rangle =\infty$
almost surely. Since $f$ has compact support, this is a contradiction
and thus the lower bound in (\ref{eq:gbounds}) strict.

In order to show that $g$ is monotone, we prove by contradiction
that for any $\alpha>0$, the super-level set $\Psi_{\alpha}\triangleq\left\{ y:\, g\left(y\right)>\alpha\right\} $
is either empty or it is an infinite interval. Let $\alpha>0$ and
assume that $\Psi_{\alpha}$ is not such. Then, from the continuity
of $g$, $\Psi_{\alpha}$ is open and thus it either has a bounded
connected component, or it is the union of two disjoint infinite intervals.
If we assume the latter, then $\left\{ y:\, g\left(y\right)\leq\alpha\right\} $
is compact and the minimum of $g$ is attained in it which contradicts
the strict lower bound of (\ref{eq:gbounds}) which was already proved.

Now assume that $\Psi_{\alpha}$ has a bounded connected component
$\left(a,b\right)$. Let $f'=\theta_{y}f$, with $y$ chosen so that
$L_{\xi}\left[\left.f'\,\right|\,\cdot\right]=g\left(\cdot\right)$.
Setting $x=\left(a+b\right)/2$, from the bounded convergence theorem,
we have that $L_{\xi}\left[\left.t\cdot f'\,\right|\, x\right]$ is
continuous in $t$ and therefore $L_{\xi}\left[\left.t\cdot f'\,\right|\, x\right]>\alpha$
for any $t\in\left[1,\delta\right]$ for some $\delta$. For each
$t\in\left[1,\delta\right]$ define $\left(a_{t},b_{t}\right)$ to
be the bounded connected component of $\left\{ y:\, L_{\xi}\left[\left.t\cdot f'\,\right|\, y\right]>\alpha\right\} $
that contains $x$ and note that, since $\alpha>0$, $a_{t}$ (respectively,
$b_{t}$) is strictly increasing (decreasing) in $t$. Hence, each
of the intervals $\left(a_{t},b_{t}\right)$ has different length.
However, each of them is, up to translation, also a connected component
of $\Psi_{\alpha}$, which, as an open subset of $\mathbb{R}$, can
only have countably many connected components. This contradicts our
assumption and implies that $g$ is monotone.

It remains to prove that the upper bound of (\ref{eq:gbounds}) is
strict which we shall prove by contradiction. Assume that the maximum
of $g$ is attained. WLOG assume that $g$ is increasing and set $\eta=\min\left\{ y:\, L_{\xi}\left[\left.f\,\right|\, y\right]=1\right\} $,
where the existence of the minimum is assured by the fact that $g$
is continuous and $\mathbb{P}\left\{ \xi\left(\mathbb{R}\right)=0\right\} =0$.
Note that
\begin{equation}
\min\left\{ y:\, L_{\xi}\left[\left.2f\,\right|\, y\right]=1\right\} =\min\left\{ y:\, L_{\xi}\left[\left.f\,\right|\, y\right]=1\right\} =\eta.\label{eq:4}
\end{equation}
Since $L_{\xi}\left[\left.f\,\right|\,\cdot\right]$ is uniquely supported,
\[
g\left(y\right)=L_{\xi}\left[\left.f\,\right|\, y-\tau_{1}\right]=L_{\xi}\left[\left.2f\,\right|\, y-\tau_{2}\right].
\]
However, from (\ref{eq:4}), $\tau_{1}=\tau_{2}$ and $L_{\xi}\left[\left.f\,\right|\, y\right]=L_{\xi}\left[\left.2f\,\right|\, y\right]$
in contradiction to the fact that $\mathbb{P}\left\{ \xi\left(\mathbb{R}\right)=0\right\} =0$.
This completes the proof.
\end{proof}
The following corollary easily follows from the lemma. In the sequel we write $\xi \gg 0$ whenever $\mathbb{P}\left\{ \xi\left(\mathbb{R}\right)>0\right\} =1$.
\begin{cor}
\label{cor:g limits}Let $\xi\gg 0$ be a point process such that $L_{\xi}\left[\left.f\,\right|\,\cdot\,\right]$ is uniquely
supported on $\left[g\right]$. Let $\tau_{n}\in\mathbb{R}$, $n=1,2,\ldots$
be a sequence such that $g\left(\cdot-\tau_{n}\right)\rightarrow h\left(\cdot\right)$
pointwise. Then either $\tau_{n}\rightarrow\pm\infty$ and $h$ is
a constant function whose value belongs to $\left\{ 0 ,1\right\} $,
or $\tau_{n}\rightarrow\tau$ for some real $\tau$ and $h\left(\cdot\right)=g\left(\cdot-\tau\right)$.
\end{cor}
Using Corollary \ref{cor:g limits} uniqueness of the support can be easily extended to functions not in $C_c^+\left(\mathbb R\right)$.
\begin{cor}
\label{cor:frz_specific_func}
Let $\xi\gg 0$ be a point process such that $L_{\xi}\left[\left.f\,\right|\,\cdot\,\right]$ is uniquely
supported on $\left[g\right]$. Let $f\geq 0$ be a measurable function on $\mathbb R$ and suppose that there exists a sequence of functions
$f_n \in C_c^+\left( \mathbb R \right)$ that converges pointwise and monotonically to $f$. If $L_{\xi}\left[\left.f\,\right|\,y\,\right]\in(0,1)$ for some $y\in\mathbb R$, then $L_{\xi}\left[\left.f\,\right|\,\cdot\,\right]\approx g\left(\cdot\right)$.
\end{cor}
\begin{proof}
The monotone convergence theorem implies that
\[
L_{\xi}\left[\left.f\,\right|\,y,\right]=\lim_{n\to\infty}L_{\xi}\left[\left.f_n\,\right|\,y,\right]= g\left( y-\tau(f_n) \right).
\]
Therefore the corollary follows from Corollary \ref{cor:g limits}.
\end{proof}
\begin{lem}
\label{lem:maxFin}
Let $\xi\gg 0$ be a point process such that $L_{\xi}\left[\left.f\,\right|\,\cdot\,\right]$ is uniquely
supported on $\left[g\right]$. If Assumption \ref{Assumption:regularity-1} holds,
then $\mathcal{M}\left(\xi\right)$ is almost
surely bounded and there exists $\tau_{\mathcal{M}}=\tau_{\mathcal{M}}\left(\xi\right)\in\mathbb{R}$ such that
\begin{align*}
L_{\xi}\left[\left.\infty\cdot\mathbf{1}_{\left(0,\infty\right)}\,\right|\, y\right] & =\mathbb{P}\left[\mathcal{\mathcal{M}\left(\xi\right)}\leq y\right]=g\left(y-\tau_{\mathcal{M}}\right).\\
\end{align*}
\end{lem}
\begin{proof}
From \eqref{eq:gbounds} it easily follows that $\mathbb P \left\{\xi([0,1])=0\right\}\in(0,1)$. Thus,
 using the previous corollary, there exists a real $\tau$ such that for any $n\in\mathbb{N}$,
\begin{align*}
  \mathbb{P}\left\{ \mathcal{M}\left(\xi\right)\geq n\right\} & \leq\sum_{i=n}^{\infty}\mathbb{P}\left\{ \xi\left(\left[i,i+1\right]\right)>0\right\} \\
 & =\sum_{i=n}^{\infty}\left(1-L_{\xi}\left[\left.\infty\cdot\mathbf{1}_{\left[0,1\right]}\,\right|\, i\, \right]\right)=\sum_{i=n}^{\infty}\left(1-g\left(y-\tau\right)\right).
\end{align*}
From Assumption \ref{Assumption:regularity-1} it easily follows that
as $n\to\infty$, the probability above converges to $0$ and therefore $\mathcal{M}\left(\xi\right)$ is almost
surely bounded. Thus, for the function $h = \infty\cdot\mathbf{1}_{\left(0,\infty\right)}$ with $y$ sufficiently high, $\mathbb P\left\{
\left\langle \theta_y h,\xi\right\rangle <\infty \right\}>0$. From the uniqueness of the support of  $L_{\xi}\left[\left.f\,\right|\,\cdot\right]$ it easily follows that $\mathbb P\left\{
\left\langle \theta_y h,\xi\right\rangle >0 \right\}>0$. Hence, applying Corollary \ref{cor:frz_specific_func} to $f$ completes
the proof.
\end{proof}
In the remainder of the section we consider the case where the intensity
measure of the point process is finite.
\begin{prop}
\label{prop:nice xi} Let $\xi\gg 0$ be a point process such that $L_{\xi}\left[\left.f\,\right|\,\cdot\,\right]$ is uniquely
supported on $\left[g\right]$. Assume for concreteness
that $g$ is increasing.
If there exists an open interval $I$ such that $\nu_{\xi}\left(I\right)<\infty$,
then there exist some constants $C,\, c>0$ such that:
\begin{enumerate}
\item \label{enu:Prop10.1}For any finite, open interval $I$, $\nu_{\xi}\left(I\right)<\infty$.
\item \label{enu:Prop10.2}The measure $\nu_{\xi}$ is absolutely continuous
relative to Lebesgue measure and its density is given by $Ce^{-cx}$.
\item \label{enu:Prop10.3}The maximum of the point process $\mathcal{M}=\mathcal{M}\left(\xi\right)$
is almost surely finite and $\mathbb{P}\left\{ \mathcal{M}\geq x\right\} \leq\frac{C}{c}e^{-cx}$,
for any $x>0$.
\item \label{enu:Prop10.4}The function $g$ satisfies (\ref{eq:21-1}),
i.e.
\[
\lim_{x\rightarrow\infty}\frac{1-g\left(x+y\right)}{1-g\left(x\right)}=\lim_{x\rightarrow\infty}\frac{\log g\left(x+y\right)}{\log g\left(x\right)}=e^{-cy}.
\]

\end{enumerate}
\end{prop}
The proposition, of course, also holds for the case where $g$ is
decreasing with the obvious adjustments to the above statements.
\begin{rem}
Note, in particular, that Proposition \ref{prop:nice xi}
implies that Assumption
\ref{Assumption:regularity-1} holds in the case that the intensity
measure is boundedly finite.
\end{rem}
\begin{proof}
Suppose that $\nu_{\xi}\left(\left(a,b\right)\right)<\infty$ for
some $a<b\in\mathbb{R}$. Set $\eta=\left(a+b\right)/2$ and $I=\left(a+\eta/2,b-\eta/2\right)$.
For any $t>0$ and $y\in\mathbb{R}$, by Corollary \ref{cor:frz_specific_func},
$L_{\xi}\left[\left.t\cdot\mathbf{1}_{I}\,\right|\, y\right]=g\left(y-\tau_{t}^{I}\right)$
for some $\tau_{t}^{I}\in\mathbb{R}$.

Recall that for any nonnegative random variable $X$, the expectation
$\mathbb{E}X$ is finite if and only if the (one-sided) derivative
$\lim_{t\searrow0}\frac{1-\mathbb{E}e^{-tX}}{t}$ exists and is finite
(see, for instance, \cite{Feller2nd}, p.435).

Thus, since for any $y\in\left[-\eta/2,\eta/2\right]$, we have $\nu_{\xi}\left(I+y\right)<\infty$,
it follows that
\begin{align}
\nu_{\xi}\left(I+y\right) & =\lim_{t\searrow0}\left[t^{-1}\left(1-\mathbb{E}\left\{ \exp\left(-t\xi\left(I+y\right)\right)\right\} \right)\right]\nonumber \\
 & =\lim_{t\searrow0}\left[t^{-1}\left(1-g\left(y-\tau_{t}^{I}\right)\right)\right].\label{eq:6}
\end{align}

Therefore, for any $y_{1}$, $y_{2}\in\left[-\eta/2,\eta/2\right]$,
\begin{equation}
\lim_{t\searrow0}\frac{1-g\left(y_{1}-\tau_{t}^{I}\right)}{1-g\left(y_{2}-\tau_{t}^{I}\right)}=\frac{\nu_{\xi}\left(I+y_{1}\right)}{\nu_{\xi}\left(I+y_{2}\right)},\label{eq:5}
\end{equation}
where $\nu_{\xi}\left(I+y\right)>0$ from the strict upper bound of
(\ref{eq:gbounds}) of Lemma \ref{lem:g monotone}, considering some
function $f\in C_{c}^{+}\left(\mathbb{R}\right)$ supported on $I+y$.

Note that for any $y\in\mathbb{R}$, $\mathbb{E}\left\{ \exp\left(-t\xi\left(I+y\right)\right)\right\} $
increases to $1$ as $t\searrow0$. Thus, since $g$ is increasing,
by Lemma \ref{lem:g monotone}, $\tau_{t}^{I}\rightarrow-\infty$.
Hence, for any $y_{1}$, $y_{2}\in\left[-\eta/2,\eta/2\right]$, the
limit of (\ref{eq:5}) depends, in fact, only on the difference $y_{2}-y_{1}$,
\begin{equation}
\lim_{t\searrow0}\frac{1-g\left(y_{1}-\tau_{t}^{I}\right)}{1-g\left(y_{2}-\tau_{t}^{I}\right)}=\lim_{x\rightarrow\infty}\frac{1-g\left(x+y_{1}\right)}{1-g\left(x+y_{2}\right)}=\lim_{x\rightarrow\infty}\frac{1-g\left(x+y_{1}-y_{2}\right)}{1-g\left(x\right)}.\label{eq:7}
\end{equation}
It follows, in particular, that for any $y\in\left[-\eta,\eta\right]$,$\lim_{x\rightarrow\infty}\frac{1-g\left(x+y\right)}{1-g\left(x\right)}$
exists.

Now, for any $y\in\left[-\eta,\eta\right]$, the limit
\[
\lim_{t\searrow0}\frac{1-g\left(y-\tau_{t}^{I}\right)}{t}=\lim_{t\searrow0}\frac{1-g\left(y-\tau_{t}^{I}\right)}{1-g\left(y/2-\tau_{t}^{I}\right)}\cdot\lim_{t\searrow0}\frac{1-g\left(y/2-\tau_{t}^{I}\right)}{t}
\]
 exists and thus $\nu_{\xi}\left(I+y\right)$ is finite, for any $y\in\left[-\eta,\eta\right]$.
By reiterating this argument we obtain that, actually, $\nu_{\xi}\left(I+y\right)$
is finite for any $y\in\mathbb{R}$. This completes the proof of (\ref{enu:Prop10.1}).

\medskip{}
Since (\ref{enu:Prop10.1}) holds, by the same arguments as in the
previous part of the proof, (\ref{eq:6}), (\ref{eq:5}) and (\ref{eq:7})
hold for any finite, open interval $I$ and any $y$, $y_{1}$, $y_{2}\in\mathbb{R}$,
with some $\tau_{t}^{I}\in\mathbb{R}$ that depends on $I$.

Define $\alpha:\mathbb{R}\rightarrow\mathbb{R}$ by the relation
\begin{equation}
e^{\alpha\left(y\right)}\triangleq\lim_{x\rightarrow\infty}\frac{1-g\left(x+y\right)}{1-g\left(x\right)}>0\label{eq:9}
\end{equation}
(where the inequality follows from (\ref{eq:5}) and the remark on
$\nu_{\xi}\left(I+y\right)>0$ right after it) .

Note that $\alpha$ is a solution to Cauchy's functional equation,
\begin{align*}
\exp\left\{ \alpha\left(y+z\right)\right\}  & =\lim_{x\rightarrow\infty}\frac{1-g\left(x+y+z\right)}{1-g\left(x+z\right)}\lim_{x\rightarrow\infty}\frac{1-g\left(x+z\right)}{1-g\left(x\right)}\\
 & =\lim_{x\rightarrow\infty}\frac{1-g\left(x+y\right)}{1-g\left(x\right)}\lim_{x\rightarrow\infty}\frac{1-g\left(x+z\right)}{1-g\left(x\right)}=\exp\left\{ \alpha\left(y\right)+\alpha\left(z\right)\right\} .
\end{align*}
Since $g$ is increasing, $\alpha$ is decreasing, and therefore $\alpha\left(y\right)=-cy$
for some $c\geq0$.

Hence, from (\ref{eq:5}) and (\ref{eq:7}), for any $y_{1}$, $y_{2}\in\mathbb{R}$,
\begin{equation}
\frac{\nu_{\xi}\left(I+y_{1}\right)}{\nu_{\xi}\left(I+y_{2}\right)}=e^{-c\left(y_{1}-y_{2}\right)}.\label{eq:8}
\end{equation}

Now, to rule out the case $c=0$, take some $f\in C_{c}^{+}\left(\mathbb{R}\right)$
such that $f\left(y\right)\leq\mathbf{1}_{\left(0,1\right)}$ and
note that since $L_{\xi}\left[\left.f\,\right|\,\cdot\,\right]$ is
uniquely supported and by Lemma \ref{lem:g monotone}, for any $\epsilon,\, t>0$
there exists $y\in\mathbb{R}$ such that
\[
\epsilon>\mathbb{E}\left\{ \exp\left(-\left\langle \theta_{y}\left(t\cdot f\right),\xi\right\rangle \right)\right\} ,
\]
thus for any $\delta>0$
\[
\epsilon>e^{-\delta}\mathbb{P}\left\{ \left\langle \theta_{y}\left(t\cdot f\right),\xi\right\rangle \leq\delta\right\}
\]
and
\[
\epsilon e^{\delta}>\mathbb{P}\left\{ \left\langle \theta_{y}f,\xi\right\rangle \leq\delta/t\right\} .
\]

Since $\epsilon$, $t$ and $\delta$ are arbitrary, it follows that
$\sup_{y\in\mathbb{R}}\nu_{\xi}\left(\left(0,1\right)+y\right)=\infty$,
which in light of (\ref{eq:8}) implies that $c\neq0$.

Substituting $\alpha\left(y\right)=-cy$ in (\ref{eq:9}) yields
\[
\lim_{x\rightarrow\infty}\frac{1-g\left(x+y\right)}{1-g\left(x\right)}=\lim_{x\rightarrow\infty}\frac{\log g\left(x+y\right)}{\log g\left(x\right)}=e^{-cy},
\]
where the equality of the two limits follows from the fact that $g\left(x+y\right)\rightarrow1$
as $x\rightarrow\infty$ for any $y\in\mathbb{R}$, and that $\left(1-e^{-w}\right)/w\rightarrow1$
as $w\rightarrow0$. This is exactly part (\ref{enu:Prop10.4}) of
the proposition.

To get rid of the nuisance of (\ref{eq:8}) being true only for open
intervals $I$, note that for any $x>0$
\[
\nu_{\xi}\left(\left(-x,x\right)\right)\geq\nu_{\xi}\left(\left(-x,-x/3\right)\right)+\nu_{\xi}\left(\left(-x/3,x/3\right)\right)+\nu_{\xi}\left(\left(x/3,x\right)\right)
\]
which, by (\ref{eq:8}), implies that
\[
2^{-1}\nu_{\xi}\left(\left(-x,x\right)\right)\geq\nu_{\xi}\left(\left(-x/3,x/3\right)\right).
\]
Hence $\nu_{\xi}\left(\left\{ 0\right\} \right)=0$. By a similar
argument for any $x\in\mathbb{R}$, $\nu_{\xi}\left(\left\{ x\right\} \right)=0$
and thus (\ref{eq:8}) hold for any finite, nontrivial interval $I$,
not necessarily open.

Using this, observe that for any $n\in\mathbb{N}$ and $x\in\mathbb{R}$,
\begin{align*}
\nu_{\xi}\left(\left[0,nx\right)\right) & =\nu_{\xi}\left(\left[0,x\right)\right)+\nu_{\xi}\left(\left[x,2x\right)\right)+\ldots+\nu_{\xi}\left(\left[\left(n-1\right)x,nx\right)\right)\\
 & =\nu_{\xi}\left(\left[0,x\right)\right)\left(1+e^{-cx}+e^{-2cx}+\ldots+e^{-\left(n-1\right)cx}\right)\\
 & =\nu_{\xi}\left(\left[0,x\right)\right)\frac{1-e^{-cnx}}{1-e^{-cx}}.
\end{align*}
Setting $C=c\nu_{\xi}\left(\left[0,1\right)\right)/\left(1-e^{-c}\right)$
and restricting to the case $x=1$ we obtain $\nu_{\xi}\left(\left[0,n\right)\right)=\frac{C}{c}\left(1-e^{-cn}\right)$.
Then, for any rational $p/q$, $p$, $q\in\mathbb{N}$, by setting
$n=q$ and $x=p/q$,
\[
\nu_{\xi}\left(\left[0,p\right)\right)=\nu_{\xi}\left(\left[0,p/q\right)\right)\frac{1-e^{-cp}}{1-e^{-cp/q}}
\]
and
\[
\nu_{\xi}\left(\left[0,p/q\right)\right)=\frac{C}{c}\left(1-e^{-cp/q}\right).
\]

Since $x\mapsto\nu_{\xi}\left(\left[0,x\right)\right)$, $x\geq0$
is monotone, we have for all $x\geq0$
\[
\nu_{\xi}\left(\left[0,x\right)\right)=\frac{C}{c}\left(1-e^{-cx}\right).
\]
Since $\nu_{\xi}\left(\left[-1,0\right)\right)=e^{c}\nu_{\xi}\left(\left[0,1\right)\right)=\frac{C}{c}\left(e^{c}-1\right)$,
by similar arguments it is seen that for $x\leq0$,
\[
\nu_{\xi}\left(\left[-x,0\right)\right)=\frac{C}{c}\left(e^{cx}-1\right).
\]

Hence, $\nu_{\xi}$ is absolutely continuous relative to Lebesgue
measure and has density $Ce^{-cx}$, and part (\ref{enu:Prop10.2})
is completed.

Lastly, note that part (\ref{enu:Prop10.3}) simply follows from Markov's
inequality where $\nu_{\xi}\left(\left[x,\infty\right)\right)$ is
easily computed from the density just derived. This completes the
proof.
\end{proof}

\section{\label{sec:(Re)construction}(Re)construction of the decoration}

Under Assumption \ref{Assumption:regularity_grl_proc} below, we show in this
section that the following limiting point process exists.
\begin{defn}
\label{def:xibar}
Suppose that the point process $\xi$ satisfies Assumption \ref{Assumption:regularity_grl_proc}. Define $\bar \xi$ to be the limit in distribution as $y\to \infty$ of $\theta_{-y}\xi$
conditioned on $\left\{ \xi\left(\left(y,\infty\right)\right)>0\right\} $.
\end{defn}
With $\bar \xi$ defined, we can define the following.
\begin{defn*}
[PP-D] Suppose that the point process $\xi$ satisfies Assumption \ref{Assumption:regularity_grl_proc}.
Define the point process $D\left(\xi\right)$ to be the point process
$\bar \xi$ translated so its maximum is at $0$, that
is, $D\left(\xi\right)\triangleq\theta_{-\mathcal{M}\left(\bar \xi\right)}\bar \xi$.
\end{defn*}

Before we proceed, to see why this process is of interest, we
investigate it in the case that $\xi$ is a DPPP.
\begin{example}
Suppose $\xi=\sum_{i\geq1}\zeta_{i}D_{i}\sim DPPP\left(e^{-cx},D\right)$
for some point process $D$ and some constant $c>0$, such that the
maximum of the decoration, $\mathcal{M}\left(D\right)$, is $0$ almost
surely.

Now consider the process $\xi$
conditioned on $\left\{ \xi\left(\left(y,\infty\right)\right)>0\right\} $.
The number of decorations that intersect $\left(y,\infty\right)$,
i.e. the number of shifted copies $\zeta_{i}D_{i}$ of $D$ that attribute
a positive measure to $\left(y,\infty\right)$, is a Poisson random
variable of parameter $c^{-1}e^{-cy}$, conditioned on being positive.
As $y\rightarrow\infty$, this random variable converges to $1$. Similarly, the number of decorations that intersect
 $\left(y-t,y\right]$ converges to $0$, for any fixed $t>0$.
Moreover, the maximum of each of the decorations that intersect $\left(y,\infty\right)$
is distributed like $y+X$ where $X\sim\mbox{exp}\left(c\right)$.
Hence, one can extract the law of the process $\theta_{X}D$,
$X\sim\mbox{exp}\left(c\right)$, by investigating $\theta_{-y}\xi$ under the conditioning above
and letting $y\rightarrow\infty$. Once we have the law of this process,
we can also recover the law of $D$.

Of course, if $\psi\sim SDPPP\left(e^{-cx},D,Z\right)$ and the right tail of the distribution of $Z$ decays fast enough,
then we should also be able to study the law of $D$ from $\bar \psi$.\qed\end{example}

In order to prove that $\bar \xi$ exists, we construct some other point process and show that
it is equal in distribution to $\bar \xi$.
The construction is done in a few stages, in each of which we define
(and, when needed, prove the existence of) a point process based on
the previous.

The following assumption does not imply that $L_{\xi}\left[\left.f\,\right|\,\cdot\,\right]$
is uniquely supported. However, as we shall see, when we do assume uniqueness, using $D\left(\xi\right)$ as the decoration to define a DPPP
we recover the original process $\xi$ up to a random shift (under
Assumption \ref{Assumption:regularity-1}). This will be the key to
the proof of Theorem \ref{thm:main}.

\begin{assm}

\label{Assumption:regularity_grl_proc} For some constant $c>0$ and
some real increasing function $g$, there exist real numbers $\tilde{\tau}_{f}=\tilde{\tau}_{f}^{g}\left(\xi\right)$
such that

\begin{align}
\lim_{x\rightarrow\infty}\frac{1-L_{\xi}\left[\left.f\,\right|\, x\right]}{1-g\left(x-\tilde{\tau}_{f}\right)} & =1,\,\,\,\,\forall f\in C_{c}^{+}\left(\mathbb{R}\right)\cup\left\{ \infty\cdot\mathbf{1}_{\left(0,\infty\right)}\right\} \label{eq:21-1-1}\\
\lim_{x\rightarrow\infty}\frac{1-g\left(x+y\right)}{1-g\left(x\right)} & =e^{-cy},\,\,\mbox{and}\label{eq:21-1-2}\\
\lim_{t\rightarrow0}\tilde{\tau}_{tf} & =\infty,\,\,\,\,\forall f\in C_{c}^{+}\left(\mathbb{R}\right).\label{eq:z}
\end{align}

\end{assm}

Since we require (\ref{eq:21-1-1}) to hold with $f=\infty\cdot\mathbf{1}_{\left(0,\infty\right)}$,
we remark that
\[
L_{\xi}\left[\left.\infty\cdot\mathbf{1}_{\left(0,\infty\right)}\,\right|\, x\right]=\mathbb{E}\left\{ \exp\left(-\left\langle \infty\cdot\mathbf{1}_{\left(x,\infty\right)},\xi\right\rangle \right)\right\} =\mathbb{P}\left\{ \mathcal{M}\left(\xi\right)\leq x\right\} .
\]
We write $\tilde{\tau}_{\mathcal{M}}$ for $\tilde{\tau}_{\infty\cdot\mathbf{1}_{\left(0,\infty\right)}}$.
\begin{rem}
\label{rk:Assumptions}By Lemma \ref{lem:maxFin}, Assumption \ref{Assumption:regularity-1}
implies Assumption \ref{Assumption:regularity_grl_proc} with $\tilde{\tau}_{f}^{g}\left(\xi\right)=\tau_{f}^{g}\left(\xi\right)$,
the shifts corresponding to the uniqueness of the support. Equation
(\ref{eq:21-1-1}) follows by definition and (\ref{eq:z}) follows
by the monotone convergence theorem.
\end{rem}
We now define the first of the point processes.
\begin{defn*}
[PP1]Given a point process $\xi$, let $\xi^{\left(y\right)}$ denote
the point process
\[
\theta_{-y}\left(\left.\xi\right|_{\left(y,\infty\right)}\right)\left(\cdot\right)=\xi\left(\left(\cdot+y\right)\cap\left(y,\infty\right)\right)
\]
conditioned on $\left\{ \xi\left(\left(y,\infty\right)\right)>0\right\} $.
\end{defn*}

\begin{defn*}
[PP2] Suppose Assumption \ref{Assumption:regularity_grl_proc} holds.
Define the point process $\xi^{\mapsto}$ as the limit in distribution
of $\xi^{\left(y\right)}$ as $y\rightarrow\infty$.

We show that the definition makes sense in the following.\end{defn*}
\begin{lem}
\label{lem:right_tail_process}Under Assumption \ref{Assumption:regularity_grl_proc},
$\xi^{\mapsto}$ as in Definition PP2 exists. Further, for any $f\in C_{c}^{+}\left(\mathbb{R}\right)$
whose support is contained in $\left(0,\infty\right)$ and any $y\geq0$,
\[
L_{\xi^{\mapsto}}\left[\left.f\,\right|\, y\right]=1-e^{-c\left(y+\tilde{\tau}_{\mathcal{M}}-\tilde{\tau}_{f}\right)}.
\]
\end{lem}
\begin{proof}
Let $f\in C_{c}^{+}\left(\mathbb{R}\right)$ be a function whose support
is contained in $\left(0,\infty\right)$ and let $\tilde{\tau}_{f}\in\mathbb{R}$
be the shift from Assumption \ref{Assumption:regularity_grl_proc}.
Then
\begin{align}
 & \negthickspace\negthickspace\negthickspace\lim_{y\rightarrow\infty}L_{\xi^{\left(y\right)}}\left[f\right]=\lim_{y\rightarrow\infty}\mathbb{E}\exp\left\{ -\left\langle f,\xi^{\left(y\right)}\right\rangle \right\} \nonumber \\
 & =\lim_{y\rightarrow\infty}\frac{\mathbb{E}\left[\exp\left\{ -\left\langle \theta_{y}f,\xi\right\rangle \right\} \,;\,\mathcal{M}\left(\xi\right)>y\right]}{\mathbb{P}\left[\mathcal{M}\left(\xi\right)>y\right]}\nonumber \\
 & =\lim_{y\rightarrow\infty}\frac{\mathbb{E}\left[\exp\left\{ -\left\langle \theta_{y}f,\xi\right\rangle \right\} \right]-\mathbb{P}\left[\mathcal{M}\left(\xi\right)\leq y\right]}{1-\mathbb{P}\left[\mathcal{M}\left(\xi\right)\leq y\right]}\nonumber \\
 & =\lim_{y\rightarrow\infty}\frac{L_{\xi}\left[\left.f\,\right|\, y\right]-L_{\xi}\left[\left.\infty\cdot\mathbf{1}_{\left(0,\infty\right)}\,\right|\, y\right]}{1-L_{\xi}\left[\left.\infty\cdot\mathbf{1}_{\left(0,\infty\right)}\,\right|\, y\right]}\nonumber \\
 & =1-\lim_{y\rightarrow\infty}\frac{1-L_{\xi}\left[\left.f\,\right|\, y\right]}{1-L_{\xi}\left[\left.\infty\cdot\mathbf{1}_{\left(0,\infty\right)}\,\right|\, y\right]}\nonumber \\
 & =1-\lim_{y\rightarrow\infty}\frac{1-g\left(y-\tilde{\tau}_{f}\right)}{1-g\left(y-\tilde{\tau}_{\mathcal{M}}\right)}=1-e^{-c\left(\tilde{\tau}_{\mathcal{M}}-\tilde{\tau}_{f}\right)}.\label{eq:14}
\end{align}

Since the point process $\xi^{\left(y\right)}$ is supported on $\left(0,\infty\right)$,
this shows that the Laplace functional $L_{\xi^{\left(y\right)}}\left[f\right]$
converges for any $f\in C_{c}^{+}\left(\mathbb{R}\right)$.

It is easy to verify that $\tilde{\tau}_{\theta_{x}f}=\tilde{\tau}_{f}-x$
for any $f\in C_{c}^{+}\left(\mathbb{R}\right)$. Therefore, for any
$f\in C_{c}^{+}\left(\mathbb{R}\right)$ whose support is contained
in $\left(0,\infty\right)$ and for any $x\geq0$,
\[
\lim_{y\rightarrow\infty}\mathbb{E}\exp\left\{ -\left\langle \theta_{x}f,\xi^{\left(y\right)}\right\rangle \right\} =1-\lim_{y\rightarrow\infty}\frac{1-g\left(y-\tilde{\tau}_{f}\right)}{1-g\left(y-\tilde{\tau}_{\mathcal{M}}\right)}=1-e^{-c\left(x+\tilde{\tau}_{\mathcal{M}}-\tilde{\tau}_{f}\right)}.
\]
Thus all that remains is to prove that for any sequence $y_{n}\in\mathbb{R}$,
$n\geq1$, such that $y_{n}\rightarrow\infty$ as $n\rightarrow\infty$,
the corresponding sequence of point processes $\xi^{\left(y_{n}\right)}$
is tight, which will imply that $\xi^{\left(y\right)}$ converges
in distribution as $y\rightarrow\infty$.

Note that, by definition, for any $y>0$, $T>1$,
\[
\mathbb{P}\left\{ \xi^{\left(y\right)}\left(\left(0,T-1/T\right)\right)>t\right\} =\frac{\mathbb{P}\left\{ \xi^{\left(y-1/T\right)}\left(\left(1/T,T\right)\right)>t\right\} }{\mathbb{P}\left\{ \xi^{\left(y-1/T\right)}\left(\left(1/T,\infty\right)\right)>0\right\} },
\]
and, using Assumption \ref{Assumption:regularity_grl_proc},
\begin{align*}
 & \negthickspace\negthickspace\negthickspace\mathbb{P}\left\{ \xi^{\left(y-1/T\right)}\left(\left(1/T,\infty\right)\right)>0\right\} =\frac{\mathbb{P}\left\{ \mathcal{M}\left(\xi\right)>y\right\} }{\mathbb{P}\left\{ \mathcal{M}\left(\xi\right)>y-1/T\right\} }\\
 & \frac{1-L_{\xi}\left[\left.\infty\cdot\mathbf{1}_{\left(0,\infty\right)}\,\right|\, y\right]}{1-L_{\xi}\left[\left.\infty\cdot\mathbf{1}_{\left(0,\infty\right)}\,\right|\, y-1/T\right]}\overset{y\rightarrow\infty}{\longrightarrow}\lim_{y\rightarrow\infty}\frac{1-g\left(y-\tau_{\mathcal{M}}\right)}{1-g\left(y-1/T-\tau_{\mathcal{M}}\right)}=e^{-c/T}.
\end{align*}
Therefore, for any sequence $\xi^{\left(y_{n}\right)}$, $n\geq1$,
as above,
\[
\lim_{t\rightarrow\infty}\limsup_{n\rightarrow\infty}\mathbb{P}\left\{ \xi^{\left(y_{n}\right)}\left(\left(0,T-1/T\right)\right)>t\right\} =e^{c/T}\cdot\lim_{t\rightarrow\infty}\limsup_{n\rightarrow\infty}\mathbb{P}\left\{ \xi^{\left(y_{n}-1/T\right)}\left(\left(1/T,T\right)\right)>t\right\} .
\]

Hence, since $\xi^{\left(y_{n}\right)}$ is supported on $\left(0,\infty\right)$,
it will be sufficient to show that (\ref{eq:30-1}) holds for any
sequence $\xi^{\left(y_{n}\right)}$, $n\geq1$, as above, with $B=\left(1/T,T\right)$
for any $T>0$.

Fix a sequence $\xi^{\left(y_{n}\right)}$ as above, fix some $T>1$,
and fix some function $f_{0}\in C_{c}^{+}\left(\mathbb{R}\right)$
with support in $\left(0,\infty\right)$ such that $f_{0}\left(x\right)\geq\mathbf{1}_{\left(1/T,T\right)}\left(x\right)$,
for any $x\in\mathbb{R}$. For any $m>0$, abbreviate $\tilde{\tau}_{mf_{0}}$
of Assumption \ref{Assumption:regularity_grl_proc} to $\tilde{\tau}_{m}$.

Note that, denoting $q_{t}\left(y\right)\triangleq\mathbb{P}\left\{ \xi^{\left(y\right)}\left(\left(1/T,T\right)\right)>t\right\} $,
for any $y$, $m$, $t>0$,
\[
\mathbb{E}\left[\exp\left\{ -\left\langle m\cdot\mathbf{1}_{\left(1/T,T\right)},\xi^{\left(y\right)}\right\rangle \right\} \right]\leq q_{t}\left(y\right)e^{-tm}+\left(1-q_{t}\left(y\right)\right)
\]
and therefore
\begin{align*}
q_{t}\left(y\right)\left(1-e^{-tm}\right) & \leq1-\mathbb{E}\left[\exp\left\{ -\left\langle m\cdot\mathbf{1}_{\left(1/T,T\right)},\xi^{\left(y\right)}\right\rangle \right\} \right]\\
 & \leq1-\mathbb{E}\left[\exp\left\{ -\left\langle mf_{0},\xi^{\left(y\right)}\right\rangle \right\} \right].
\end{align*}

Thus, for any $m>0$,
\[
\lim_{t\rightarrow\infty}\limsup_{n\rightarrow\infty}q_{t}\left(y_{n}\right)=\lim_{t\rightarrow\infty}\limsup_{n\rightarrow\infty}q_{t}\left(y_{n}\right)\left(1-e^{-tm}\right)\leq e^{-c\left(\tilde{\tau}_{\mathcal{M}}-\tilde{\tau}_{m}\right)},
\]
and therefore, by (\ref{eq:z}),
\[
\lim_{t\rightarrow\infty}\limsup_{n\rightarrow\infty}q_{t}\left(y_{n}\right)=0,
\]
which proves tightness.\end{proof}
\begin{cor}
\label{cor:max_of_tail}Under Assumption \ref{Assumption:regularity_grl_proc}
the maximum of $\xi^{\mapsto}$, $\mathcal{M}\left(\xi^{\mapsto}\right)$,
is an exponential random variable with parameter $c$.\end{cor}
\begin{proof}
Approximating $\infty\cdot\mathbf{1}_{\left(0,\infty\right)}$ by
an increasing sequence of functions $f_{n}\in C_{c}^{+}\left(\mathbb{R}\right)$
whose support is contained in $\left(0,\infty\right)$ easily yields
\[
\mathbb{P}\left[\mathcal{M}\left(\xi^{\mapsto}\right)\leq y\right]=L_{\xi^{\mapsto}}\left[\left.\infty\cdot\mathbf{1}_{\left(0,\infty\right)}\,\right|\, y\right]=1-e^{-c\left(y+\tilde{\tau}_{\mathcal{M}}-\tilde{\tau}_{\mathcal{M}}\right)}=1-e^{-cy}.
\]
\end{proof}
\begin{defn*}
[PP3] Suppose Assumption \ref{Assumption:regularity_grl_proc} holds.
For any $y>0$ define $\xi_{\left(y\right)}$ to be the point process
 $\theta_{-y}\xi^{\mapsto}$ conditioned on $\left\{ \xi^{\mapsto}\left(\left(y,\infty\right)\right)>0\right\} $
($\xi_{\left(y\right)}$ is supported on $\left(-y,\infty\right)$).\end{defn*}
\begin{lem}
\label{lem:19}Under Assumption \ref{Assumption:regularity_grl_proc},
for any $0\leq y,\, t$ and any $f\in C_{c}^{+}\left(\mathbb{R}\right)$
whose support is contained in $\left(0,\infty\right)$,
\[
L_{\xi_{\left(y\right)}}\left[f\right]=1-e^{-c\left(\tilde{\tau}_{\mathcal{M}}-\tilde{\tau}_{f}\right)}\mbox{\,\, and\,\,}L_{\xi_{\left(y+t\right)}}\left[\theta_{-y}f\right]\mbox{ is independent of }t.
\]
Thus, in particular, for any $0\leq y,\, t$,
\begin{equation}
\xi_{\left(y\right)}\overset{d}{=}\left.\xi_{\left(y+t\right)}\right|_{\left(-y,\infty\right)}.\label{eq:26}
\end{equation}
\end{lem}
\begin{proof}
By Lemma \ref{lem:right_tail_process} and Corollary \ref{cor:max_of_tail},
for any $0\leq y$ and any $f\in C_{c}^{+}\left(\mathbb{R}\right)$
whose support is contained in $\left(0,\infty\right)$,
\begin{align*}
L_{\xi_{\left(y\right)}}\left[f\right] & =\mathbb{E}\left[\left.\exp\left\{ -\left\langle \theta_{y}f,\xi^{\mapsto}\right\rangle \right\} \,\right|\,\mathcal{M}\left(\xi^{\mapsto}\right)>y\right]\\
 & =\frac{\mathbb{E}\left[\exp\left\{ -\left\langle \theta_{y}f,\xi^{\mapsto}\right\rangle \right\} \,;\,\mathcal{M}\left(\xi^{\mapsto}\right)>y\right]}{\mathbb{P}\left[\mathcal{M}\left(\xi^{\mapsto}\right)>y\right]}\\
 & =\frac{\mathbb{E}\left[\exp\left\{ -\left\langle \theta_{y}f,\xi^{\mapsto}\right\rangle \right\} \right]-\mathbb{P}\left[\mathcal{M}\left(\xi^{\mapsto}\right)\leq y\right]}{\mathbb{P}\left[\mathcal{M}\left(\xi^{\mapsto}\right)>y\right]}\\
 & =\frac{1-e^{-c\left(y+\tilde{\tau}_{\mathcal{M}}-\tilde{\tau}_{f}\right)}-\left(1-e^{-cy}\right)}{e^{-cy}}\\
 & =1-e^{-c\left(\tilde{\tau}_{\mathcal{M}}-\tilde{\tau}_{f}\right)},
\end{align*}
and therefore, for any $0\leq y$, $t$,
\begin{equation}
\left.\xi_{\left(y\right)}\right|_{\left(0,\infty\right)}\overset{d}{=}\left.\xi_{\left(y+t\right)}\right|_{\left(0,\infty\right)}.\label{eq:25}
\end{equation}

Now, for any $0\leq y$, $t$ and any $f\in C_{c}^{+}\left(\mathbb{R}\right)$
whose support is contained in $\left(0,\infty\right)$,
\begin{align*}
L_{\xi_{\left(y+t\right)}}\left[\theta_{-t}f\right] & =\mathbb{E}\left[\left.\exp\left\{ -\left\langle \theta_{y}f,\xi^{\mapsto}\right\rangle \right\} \,\right|\,\mathcal{M}\left(\xi^{\mapsto}\right)>y+t\right]\\
 & =\frac{\mathbb{E}\left[\exp\left\{ -\left\langle \theta_{y}f,\xi^{\mapsto}\right\rangle \right\} \,;\,\mathcal{M}\left(\xi^{\mapsto}\right)>y+t\right]}{\mathbb{P}\left[\mathcal{M}\left(\xi^{\mapsto}\right)>y+t\right]}\\
 & =\frac{\mathbb{P}\left[\mathcal{M}\left(\xi^{\mapsto}\right)>y\right]}{\mathbb{P}\left[\mathcal{M}\left(\xi^{\mapsto}\right)>y+t\right]}\mathbb{E}\left[\left.\exp\left\{ -\left\langle \theta_{y}f,\xi^{\mapsto}\right\rangle \right\} \cdot\mathbf{1}_{\left\{ \mathcal{M}\left(\xi^{\mapsto}\right)>y+t\right\} }\right|\mathcal{M}\left(\xi^{\mapsto}\right)>y\right]\\
 & =e^{ct}\mathbb{E}\left[\left.\exp\left\{ -\left\langle f,\theta_{-y}\xi^{\mapsto}\right\rangle \right\} \cdot\mathbf{1}_{\theta_{-y}\xi^{\mapsto}\left(\left(t,\infty\right)\right)}\right|\mathcal{M}\left(\xi^{\mapsto}\right)>y\right].
\end{align*}
According to (\ref{eq:25}) the last expression is independent of
$y$, which completes the proof.
\end{proof}
Equation (\ref{eq:26}) allows us to define the following.
\begin{defn*}
[PP4] Suppose Assumption \ref{Assumption:regularity_grl_proc} holds.
Define $\xi^{\leftrightarrow}$ as the limit in distribution of $\xi_{\left(y\right)}$
as $y\rightarrow\infty$.
\end{defn*}
We are now ready to prove that the limiting point process $\bar \xi$ of definition \ref{def:xibar} exists.
\begin{lem}
\label{lem:PPequiv}
Under Assumption \ref{Assumption:regularity_grl_proc}, the point process $\bar \xi$ exists and is equal in distribution to $\xi^{\leftrightarrow}$.
\end{lem}
\begin{proof}
Let $n\geq0$, let $A_i\subset \mathbb R$, $i\leq n$, be Borel sets, and let $k_i\in \mathbb R$, $i\leq n$. We need to show that
\begin{equation}
\label{eq:PPequiv}
\lim_{y\to\infty}\frac{\mathbb P\left\{ \xi(A_i+y)\geq k_i,\, i\leq n,\, \xi((y,\infty))\right\}}{\mathbb P\left\{ \xi((y,\infty))\right\}} =
\mathbb P\left\{ \xi^{\leftrightarrow}(A_i)\geq k_i,\, i\leq n \right\}.
\end{equation}

From Definition PP3 and Definition PP4,
\begin{equation}
\label{eq:Mar31_1}
\mathbb P\left\{ \xi^{\leftrightarrow}(A_i)\geq k_i,\, i\leq n \right\} =
\lim_{y\to\infty}\frac{\mathbb P\left\{ \xi^{\mapsto}(A_i+y)\geq k_i,\, i\leq n,\, \xi^{\mapsto}((y,\infty))\right\}}{\mathbb P\left\{ \xi^{\mapsto}((y,\infty))\right\}}.
\end{equation}

For large enough $y$, from Definition PP1 and Definition PP2,
\begin{align*}
& \!\!\!\!\!\! \mathbb P\left\{ \xi^{\mapsto}(A_i+y)\geq k_i,\, i\leq n,\, \xi^{\mapsto}((y,\infty))\right\} \\
& = \lim_{t\to\infty} \mathbb P\left\{ \xi^{\mapsto}(A_i+y)\geq k_i,\, i\leq n,\, \xi^{\mapsto}((y,y+t))\right\} \\
& = \lim_{t\to\infty}\lim_{r\to\infty} \frac{\mathbb P\left\{ \xi(A_i+y+r)\geq k_i,\, i\leq n,\, \xi((y+r,y+r+t))\right\}}{\mathbb P\left\{ \xi((r,\infty))\right\}}.
\end{align*}

Note that, using Assumption \ref{Assumption:regularity_grl_proc},
\begin{align*}
& \lim_{t\to\infty}\lim_{r\to\infty} \frac{\mathbb P\left\{ \xi(A_i+y+r)\geq k_i,\, i\leq n,\, \xi((y+r,\infty))\right\}}{\mathbb P\left\{ \xi((r,\infty))\right\}}\\
& - \lim_{t\to\infty}\lim_{r\to\infty} \frac{\mathbb P\left\{ \xi(A_i+y+r)\geq k_i,\, i\leq n,\, \xi((y+r,y+r+t))\right\}}{\mathbb P\left\{ \xi((r,\infty))\right\}} \\
& \quad \leq \lim_{t\to\infty}\lim_{r\to\infty} \frac{\mathbb P\left\{ \xi((y+r+t,\infty))\right\}}{\mathbb P\left\{ \xi((r,\infty))\right\}} = \lim_{t\to\infty}\lim_{r\to\infty}\frac{1-g(y+r+t-\tilde{\tau}_{\mathcal{M}})}{1-g(r-\tilde{\tau}_{\mathcal{M}})}\\
& \quad = \lim_{t\to\infty}e^{-c(y+t)}=0,
\end{align*}
and therefore
\begin{align*}
& \!\!\!\!\!\! \mathbb P\left\{ \xi^{\mapsto}(A_i+y)\geq k_i,\, i\leq n,\, \xi^{\mapsto}((y,\infty))\right\} \\
& = \lim_{r\to\infty} \frac{\mathbb P\left\{ \xi(A_i+y+r)\geq k_i,\, i\leq n,\, \xi((y+r,\infty))\right\}}{\mathbb P\left\{ \xi((r,\infty))\right\}}.
\end{align*}
Setting $k_i=0$, we similarly obtain
\begin{align*}
\mathbb P\left\{ \xi^{\mapsto}((y,\infty))\right\}  = \lim_{r\to\infty} \frac{\mathbb P\left\{  \xi((y+r,\infty))\right\}}{\mathbb P\left\{ \xi((r,\infty))\right\}}>0.
\end{align*}

Substituting in \eqref{eq:Mar31_1} yields
\begin{equation*}
\mathbb P\left\{ \xi^{\leftrightarrow}(A_i)\geq k_i,\, i\leq n \right\} =
\lim_{y\to\infty}\lim_{r\to\infty}\frac{\mathbb P\left\{ \xi(A_i+y+r)\geq k_i,\, i\leq n,\, \xi((y+r,\infty))\right\}}{\mathbb P\left\{  \xi((y+r,\infty))\right\}},
\end{equation*}
which implies \eqref{eq:PPequiv} and completes the proof.
\end{proof}
From (\ref{eq:26}), we have that $\left.\xi^{\leftrightarrow}\right|_{\left(-y,\infty\right)}\overset{d}{=}\xi_{\left(y\right)}$
for any $0\leq y$. The following then easily follows.
\begin{cor}
\label{cor:20}Suppose Assumption \ref{Assumption:regularity_grl_proc}
holds. Then for any $f\in C_{c}^{+}\left(\mathbb{R}\right)$,
\begin{equation}
\mathbb{E}\left[\left.\exp\left\{ -\left\langle f,\theta_{-y}\xi^{\leftrightarrow}\right\rangle \right\} \,\right|\,\mathcal{M}\left(\xi^{\leftrightarrow}\right)>y\right]\mbox{ \,\ is independent of }y>0,\label{eq:27}
\end{equation}
and $\mathcal{M}\left(\xi^{\leftrightarrow}\right)$ is an exponential
random variable with parameter $c$, and therefore the point process
defined as $\theta_{-y}\xi^{\leftrightarrow}$ conditioned on the
event $\left\{ \mathcal{M}\left(\xi^{\leftrightarrow}\right)>y\right\} $
is independent of $y>0$.\end{cor}
\begin{proof}
Let $f\in C_{c}^{+}\left(\mathbb{R}\right)$ and $y\geq0$. Then for
$x$ large enough such that the support $\theta_{y}f$ is contained
in $\left(-x,\infty\right)$,
\begin{align*}
 & \negthickspace\negthickspace\negthickspace\mathbb{E}\left[\left.\exp\left\{ -\left\langle f,\theta_{-y}\xi^{\leftrightarrow}\right\rangle \right\} \,\right|\,\mathcal{M}\left(\xi^{\leftrightarrow}\right)>y\right]\\
 & =\mathbb{E}\left[\left.\exp\left\{ -\left\langle f,\theta_{-y}\xi_{\left(x\right)}\right\rangle \right\} \,\right|\,\mathcal{M}\left(\xi_{\left(x\right)}\right)>y\right]\\
 & =\frac{\mathbb{E}\left[\exp\left\{ -\left\langle f,\theta_{-y}\xi_{\left(x\right)}\right\rangle \right\} \cdot\mathbf{1}_{\left\{ \mathcal{M}\left(\xi_{\left(x\right)}\right)>y\right\} }\right]}{\mathbb{P}\left[\mathcal{M}\left(\xi_{\left(x\right)}\right)>y\right]}\\
 & =\frac{\mathbb{E}\left[\left.\exp\left\{ -\left\langle f,\theta_{-y}\theta_{-x}\xi^{\mapsto}\right\rangle \right\} \cdot\mathbf{1}_{\left\{ \mathcal{M}\left(\theta_{-x}\xi^{\mapsto}\right)>y\right\} }\,\right|\,\mathcal{M}\left(\xi^{\mapsto}\right)>x\right]}{\mathbb{P}\left[\left.\mathcal{M}\left(\theta_{-x}\xi^{\mapsto}\right)>y\,\right|\,\mathcal{M}\left(\xi^{\mapsto}\right)>x\right]}\\
 & =\frac{\mathbb{E}\left[\exp\left\{ -\left\langle f,\theta_{-\left(x+y\right)}\xi^{\mapsto}\right\rangle \right\} \cdot\mathbf{1}_{\left\{ \mathcal{M}\left(\xi^{\mapsto}\right)>x+y\right\} }\right]}{\mathbb{P}\left[\mathcal{M}\left(\theta_{-x}\xi^{\mapsto}\right)>x+y\right]}\\
 & =\mathbb{E}\left[\left.\exp\left\{ -\left\langle f,\theta_{-\left(x+y\right)}\xi^{\mapsto}\right\rangle \right\} \,\right|\,\mathcal{M}\left(\xi^{\mapsto}\right)>x+y\right].
\end{align*}

Thus, for two different $y_{1},\, y_{2}>0$, choosing $x_{1}$ and
$x_{2}$ large enough and such that $x_{1}+y_{1}=x_{2}+y_{2}$, we
have
\begin{align*}
 & \negthickspace\negthickspace\negthickspace\mathbb{E}\left[\left.\exp\left\{ -\left\langle f,\theta_{-y_{1}}\xi^{\leftrightarrow}\right\rangle \right\} \,\right|\,\mathcal{M}\left(\xi^{\leftrightarrow}\right)>y_{1}\right]\\
 & =\mathbb{E}\left[\left.\exp\left\{ -\left\langle f,\theta_{-\left(x_{1}+y_{1}\right)}\xi^{\mapsto}\right\rangle \right\} \,\right|\,\mathcal{M}\left(\xi^{\mapsto}\right)>x_{1}+y_{1}\right]\\
 & =\mathbb{E}\left[\left.\exp\left\{ -\left\langle f,\theta_{-y_{2}}\xi^{\leftrightarrow}\right\rangle \right\} \,\right|\,\mathcal{M}\left(\xi^{\leftrightarrow}\right)>y_{2}\right].
\end{align*}

Relying on the fact that $\left.\xi^{\leftrightarrow}\right|_{\left(-y,\infty\right)}=\xi_{-y\mapsto}$
with $y=0$ we have $\left.\xi^{\leftrightarrow}\right|_{\left(0,\infty\right)}=\xi^{\mapsto}$
and thus, from Corollary \ref{cor:max_of_tail}, $\mathcal{M}\left(\xi^{\leftrightarrow}\right)$
is exponentially distributed with parameter $c$.
\end{proof}

Note that by Lemma \ref{lem:PPequiv} and Definition PP-D, $D\left(\xi\right)\triangleq\theta_{-\mathcal{M}\left(\xi^{\leftrightarrow}\right)}\xi^{\leftrightarrow}$. Corollary \ref{cor:20} immediately gives the following.
\begin{cor}
The point process $D\left(\xi\right)$ is independent of the maximum
$\mathcal{M}\left(\xi^{\leftrightarrow}\right)$.\end{cor}
\begin{proof}
It is sufficient to show that $\mathbb{E}\left[\left.\exp\left\{ -\left\langle f,D\left(\xi\right)\right\rangle \right\} \,\right|\,\mathcal{M}\left(\xi^{\leftrightarrow}\right)>y\right]$
is independent of $y\geq0$ for any $f\in C_{c}^{+}\left(\mathbb{R}\right)$.
This follows from Corollary \ref{cor:20} since

\begin{align*}
 & \negthickspace\negthickspace\negthickspace\mathbb{E}\left[\left.\exp\left\{ -\left\langle f,D\left(\xi\right)\right\rangle \right\} \,\right|\,\mathcal{M}\left(\xi^{\leftrightarrow}\right)>y\right]\\
 & =\mathbb{E}\left[\left.\exp\left\{ -\left\langle \theta_{\mathcal{M}\left(\xi^{\leftrightarrow}\right)}f,\xi^{\leftrightarrow}\right\rangle \right\} \,\right|\,\mathcal{M}\left(\xi^{\leftrightarrow}\right)>y\right]\\
 & =\mathbb{E}\left[\left.\exp\left\{ -\left\langle \theta_{\mathcal{M}\left(\xi^{\leftrightarrow}\right)-y}f,\theta_{-y}\xi^{\leftrightarrow}\right\rangle \right\} \,\right|\,\mathcal{M}\left(\xi^{\leftrightarrow}\right)>y\right]\\
 & =\mathbb{E}\left[\left.\exp\left\{ -\left\langle \theta_{\mathcal{M}\left(\theta_{-y}\xi^{\leftrightarrow}\right)}f,\theta_{-y}\xi^{\leftrightarrow}\right\rangle \right\} \,\right|\,\mathcal{M}\left(\xi^{\leftrightarrow}\right)>y\right].
\end{align*}

\end{proof}
The two last corollaries directly give the following.
\begin{cor}
\label{cor:22}If Assumption \ref{Assumption:regularity_grl_proc}
holds, then $\xi^{\leftrightarrow}\overset{d}{=}\theta_{X}D\left(\xi\right)$,
where $X\sim\mbox{exp}\left(c\right)$ is independent of $D\left(\xi\right)$.
\end{cor}
The main result of this section can now be proven.
\begin{lem}
\label{lem:DPPP(xi)}Suppose Assumption \ref{Assumption:regularity_grl_proc}
holds. If
\[
\psi\left(\xi\right)\sim DPPP\left(e^{-cx}dx,\theta_{-\tilde{\tau}_{\mathcal{M}}^{g}\left(\xi\right)-c^{-1}\log c}D\left(\xi\right)\right),
\]
then $L_{\psi\left(\xi\right)}\left[\left.f\,\right|\,\cdot\,\right]$
is uniquely support on the class $\left[{\rm {Gum}}\left(cy\right)\right]$
with shifts $\tau_{f}^{{\rm {Gum}}\left(cy\right)}\left(\psi\left(\xi\right)\right)=\tilde{\tau}_{f}^{g}\left(\xi\right)$
(i.e., the shifts of Assumption \ref{Assumption:regularity_grl_proc}),
for any $f\in C_{c}^{+}\left(\mathbb{R}\right)$.\end{lem}
\begin{proof}
Note that the lemma follows if we show that for $\psi'=\psi'\left(\xi\right)\sim DPPP\left(e^{-cx},D\left(\xi\right)\right)$,
$L_{\psi'}\left[\left.f\,\right|\,\cdot\,\right]$ is uniquely supported
on the class $\left[{\rm {Gum}}\left(cy\right)\right]$ with shifts
\begin{equation}
\tau_{f}\left(\psi'\right)=\tilde{\tau}_{f}\left(\xi\right)-\tilde{\tau}_{\mathcal{M}}\left(\xi\right)-c^{-1}\log c,\label{eq:29}
\end{equation}
for any $f\in C_{c}^{+}\left(\mathbb{R}\right)$.

From the calculations of Section \ref{sec:dppp_to_frz}, $L_{\psi'}\left[\left.f\,\right|\,\cdot\,\right]$
is uniquely supported on the class $\left[{\rm {Gum}}\left(cy\right)\right]$
and the shifts are given, for any $f\in C_{c}^{+}\left(\mathbb{R}\right)$,
by
\begin{equation}
\tau_{f}\left(\psi'\right)=\tau_{f,c}=c^{-1}\log\left(-\int_{-\infty}^{\infty}e^{-ct}\left(L_{D\left(\xi\right)}\left[\left.f\,\right|\,-t\right]-1\right)dt\right).\label{eq:28}
\end{equation}

Note that it is sufficient to verify that (\ref{eq:29}) holds for
any $f\in C_{c}^{+}\left(\mathbb{R}\right)$ such that $\inf\left\{ x\in\mathbb{R}:\, f\left(x\right)>0\right\} =0$.
If $f$ is such a function, then, since $\mathcal{M}\left(D\left(\xi\right)\right)=0$
a.s.,
\begin{align*}
 & \negthickspace\negthickspace\negthickspace\negthickspace\int_{-\infty}^{\infty}e^{-ct}\left(L_{D\left(\xi\right)}\left[\left.f\,\right|\,-t\right]-1\right)dt=\int_{0}^{\infty}e^{-ct}\left(L_{D\left(\xi\right)}\left[\theta_{-t}f\right]-1\right)dt\\
 & =\int_{0}^{\infty}e^{-ct}\left(L_{\theta_{t}D\left(\xi\right)}\left[f\right]-1\right)dt=c^{-1}\left(L_{\xi^{\leftrightarrow}}\left[f\right]-1\right)=c^{-1}\left(L_{\xi_{\left(0\right)}}\left[f\right]-1\right)=-c^{-1}e^{-c\left(\tilde{\tau}_{\mathcal{M}}\left(\xi\right)-\tilde{\tau}_{f}\left(\xi\right)\right)},
\end{align*}
where we used Corollary \ref{cor:22} and Lemma \ref{lem:19}.

Substituting this in (\ref{eq:28}) yields (\ref{eq:29}).
\end{proof}
Using Lemma \ref{lem:DPPP(xi)}, Corollary \ref{cor:exp-stb} can
be now easily proved.

\subsection*{Proof of Corollary \ref{cor:exp-stb}}

We restrict to the case $c=1$, for $c\neq1$ the proof is similar.
The implication \ref{enu:exp-stb-1}$\Rightarrow$\ref{enu:exp-stb-2}
is immediate. We will show that \ref{enu:exp-stb-2}$\Rightarrow$\ref{enu:exp-stb-3}$\Rightarrow$\ref{enu:exp-stb-1}.

Assuming \ref{enu:exp-stb-2}, we have that, for any $f\in C_{c}^{+}\left(\mathbb{R}\right)$
and $y>0$, $a,\, b\in\mathbb{R}$, such that $e^{a}+e^{b}=1$,
\begin{align*}
L_{\xi}\left[\left.f\,\right|\,-\log y\right] & =L_{\theta_{a}\xi'+\theta_{b}\xi''}\left[\left.f\,\right|\,-\log y\right]\\
 & =L_{\xi}\left[\left.f\,\right|\,-\log y-a\right]\cdot L_{\xi}\left[\left.f\,\right|\,-\log y-b\right]\\
 & =L_{\xi}\left[\left.f\,\right|\,-\log\left(e^{a}y\right)\right]\cdot L_{\xi}\left[\left.f\,\right|\,-\log\left(e^{b}y\right)\right],
\end{align*}
where $\xi'$ and $\xi''$ are i.i.d copies of $\xi$.

It follows that if $L_{\xi}\left[\left.f\,\right|\, y\right]=1$ for some $y$, then 
$L_{\xi}\left[\left.f\,\right|\, y'\right]=1$ for all $y'>y$. Thus, if there exists some $y$ such that 
$L_{\xi}\left[\left.f\,\right|\, y\right]=1$, then $\xi((t,\infty))=0$ a.s. for some $t\in \mathbb R$. 
From exponential-stability (with $a=b=-\log 2$), if this occurs then $\xi((t-\log 2,\infty))=0$ a.s. By repeating the same argument, we obtain that 
$L_{\xi}\left[\left.f\,\right|\, y\right]=1$ for some $y$ only if $\mathbb P\left\{\xi(\mathbb R)>0\right\}=0$. 
Since we assumed it is not so, $L_{\xi}\left[\left.f\,\right|\, y\right]\in(0,1)$, for all $y\in\mathbb R$.  

Defining $\varphi\left(x\right)=\log L_{\xi}\left[\left.f\,\right|\,-\log x\right]$,
for any $x,\, y>0$, $\varphi\left(x+y\right)=\varphi\left(x\right)+\varphi\left(y\right)$,
i.e., $\varphi$ is a solution of Cauchy's functional equation \cite{cauchy}.
Since $L_{\xi}\left[\left.f\,\right|\, y\right]\in (0,1)$,
$\varphi(x)$<0, and therefore $\varphi\left(x\right)=\varphi\left(1\right)x$.
Hence
\[
L_{\xi}\left[\left.f\,\right|\, y\right]=\exp\left\{ -e^{-\left(y-\tau_{f}\right)}\right\} \approx\mbox{Gum}\left(y\right),
\]
which proves \ref{enu:exp-stb-3}.

Lastly, assume that \ref{enu:exp-stb-3} holds. Note that $\xi$ satisfies
Assumption \ref{Assumption:regularity-1} and thus, as noted in Remark
\ref{rk:Assumptions}, it satisfies Assumption \ref{Assumption:regularity_grl_proc}.
Let $\psi\left(\xi\right)$ be the corresponding DPPP from Lemma \ref{lem:DPPP(xi)}.
Observe that the Laplace functionals of $\psi\left(\xi\right)$ and
$\xi$ coincide on $C_{c}^{+}\left(\mathbb{R}\right)$ and thus both
processes have the same law. That is, \ref{enu:exp-stb-1} is satisfied
with $\xi\sim DPPP\left(e^{-x}dx,\hat{D}\left(\xi\right)\right)$
and the proof is completed.\qed

\section{\label{sec:mainpf}Proof of Theorem \ref{thm:main}: the direct part}

We shall prove each of the assertions of the direct part of Theorem
\ref{thm:main} separately. The notation $X\perp Y$ will be used
whenever two random variables $X$ and $Y$ are independent. We
begin with the following corollary of Lemma \ref{lem:DPPP(xi)}.
\begin{cor}
\label{cor:g=00003DG*Z}Part $(a)$ of the direct part of Theorem
\ref{thm:main}  holds.\end{cor}
\begin{proof}
First, recall that, as stated in Remark \ref{rk:Assumptions}, Assumption
\ref{Assumption:regularity_grl_proc} holds with shifts $\tilde{\tau}_{f}^{g}\left(\xi\right)$
equal to the shifts of Assumption \ref{Assumption:regularity-1},
$\tau_{f}^{g}\left(\xi\right)$.

Suppose (SUS) holds with some random variable $Z$ and with $c=c_{\xi}$.
Let $\psi=\psi\left(\xi\right)$ be the corresponding DPPP from Lemma
\ref{lem:DPPP(xi)}, which by the lemma satisfies $\tau_{f}^{{\rm {Gum}}\left(cy\right)}\left(\psi\right)=\tau_{f}^{g}\left(\xi\right)$.
Define $\psi$ and $Z$ on the same probability space such that $Z\perp\psi$
and consider the point process $\theta_{Z}\psi$. Observe that, by
conditioning on $Z$, the Laplace functional of $\theta_{Z}\psi$
is given, for any $f\in C_{c}^{+}\left(\mathbb{R}\right)$, by
\begin{align*}
L_{\theta_{Z}\psi}\left[f\right] & =\int L_{\theta_{z}\psi}\left[f\right]d\mu_{Z}\left(z\right)=\int L_{\psi}\left[\left.f\,\right|\,-z\right]d\mu_{Z}\left(z\right)\\
 & =\int{\rm {Gum}}\left(c\left(-z-\tau_{f}\right)\right)d\mu_{Z}\left(z\right)=g\left(-\tau_{f}\right)=L_{\xi}\left[f\right],
\end{align*}
that is, the Laplace functionals of $\theta_{Z}\psi$ and $\xi$ coincide
on $C_{c}^{+}\left(\mathbb{R}\right)$. This implies $\xi\sim SDPPP\left(e^{-cx}dx,\hat{D}\left(\xi\right),Z\right)$
as required.
\end{proof}
We now need to treat the case where the only assumption on $g$ is
Assumption \ref{Assumption:regularity-1}.
In this case, in order to obtain a process
that satisfies (SUS) we consider the process $\theta_{S}\xi$, where the
a random variable $S$ has itself a Gumbel distribution and is independent of $\xi$. Doing so easily yields
the following.
\begin{lem}
\label{lem:23}Part $(b)$ of the direct part of Theorem \ref{thm:main}
holds.\end{lem}
\begin{proof}
From Lemma \ref{lem:DPPP(xi)} the DPPP $\psi=\psi\left(\xi\right)$
has shifts $\tau_{f}^{{\rm {Gum}}^{\left(c\right)}}\left(\psi\right)=\tau_{f}^{g}\left(\xi\right)$,
where ${\rm {Gum}}^{\left(c\right)}\left(y\right)\triangleq{\rm {Gum}}\left(cy\right)$.
Observe that, if $Z_{G}$ and $Z_{g}$ are random variables with cumulative
distribution functions ${\rm {Gum}}^{\left(c\right)}$ and $g$ independent
of $\xi$ and $\psi$, respectively, then for any $f\in C_{c}^{+}\left(\mathbb{R}\right)$,
\[
L_{\theta_{Z_{{\rm G}}}\xi}\left[f\right]={\rm {Gum}}^{\left(c\right)}*g\left(-\tau_{f}\right)=g*{\rm {Gum}}^{\left(c\right)}\left(-\tau_{f}\right)=L_{\theta_{Z_{g}}\psi}\left[f\right],
\]
that is, the Laplace functionals of $\theta_{Z_{G}}\xi$ and $\theta_{Z_{g}}\psi$
coincide on $C_{c}^{+}\left(\mathbb{R}\right)$ and therefore $\theta_{Z_{G}}\xi$
and $\theta_{Z_{g}}\psi$ have the same law.

Denote by $\mathcal{N}$ the space of positive, locally finite, counting
measures on $\mathbb{R}$ (cf. Section \ref{sec:Appendix}). Defining
the measurable function $h:\,\mathcal{N}\times\mathbb{R}\rightarrow\mathcal{N}$,
$\left(\zeta,Z\right)\mapsto\theta_{Z}\zeta$,  we have $\theta_{Z_{g}}\psi=h\left(\psi,Z_{g}\right)$
and $\theta_{Z_{{\rm G}}}\xi\overset{d}{=}h\left(\psi,Z_{g}\right)$.
Hence, according to the transfer principle  \cite[Corollary 6.11]{Kallenberg2nd} (and due
to the fact that $\left(\mathcal{N},\mathcal{A}\right)$ is a Borel
space), we can in fact couple the random variables $Z_{G}$, $Z_{g}$,
$\xi$ and $\psi$ such that $Z_{G}\perp\xi$, $Z_{g}\perp\psi$,
and $\theta_{Z_{G}}\xi=\theta_{Z_{g}}\psi$ almost surely. On this
space, we also have
\[
\xi=\theta_{Z_{g}-Z_{G}}\psi\triangleq\theta_{Z}\psi\,,\mbox{\,\,\,\ almost surely}.
\]
This completes the proof. Note that $Z_{G}$, and hence $Z$, may
depend on $\psi$.
\end{proof}
What remains is to show that under the assumptions of parts $(c)$ and $(d)$ of the direct part of Theorem
\ref{thm:main}, (SUS) holds. We shall begin with part $(c)$, the proof of which will follow from the following result implying that the structure of
the function $g$ specified in (SUS) is determined
by the freezing phenomenon discussed in the introduction.
\begin{lem}
\label{lem:150414-1}Let $\xi\gg 0$ be a point process and let $c>0$ such that
\[
\mbox{for any }\beta>c:\,\, L_{\xi}\left[\left.e^{\beta x}\,\right|\,\cdot\,\right]\approx g\left(\cdot\right),
\]
where $g:\mathbb{R}\rightarrow\left(0,1\right)$ is a function for which $\sup_{y\in\mathbb R}g(y)=1$.
Then there exists a random variable on $\mathbb{R}$, $Z$, with law
$\mu_{Z}$, such that
\[
g\left(y\right)=\int{\rm {Gum}}\left(\bar{\beta}\left(y-z\right)\right)d\mu_{Z}\left(z\right)=\mathbb{E}\left\{ {\rm {Gum}}\left(\bar{\beta}\left(y-Z\right)\right)\right\} .
\]
\end{lem}
\begin{proof}
Fix some $\beta'>\beta>c$.
Let $\eta$ be the random measure given by
\[
\eta=\sum_{i\geq1}\theta_{\xi_{i}}P_{i},
\]
where $\xi=\sum_{i\geq1}\delta_{\xi_{i}}$ and where $P_{i}$ are
Poisson point processes of intensity $e^{-\beta x}dx$, independent
of each other and $\xi$.

Since
\[
\mathbb E\exp\left\{-\left\langle \theta_{y}e^{\beta x},\xi\right\rangle \right\} \leq 1-\mathbb P\left\{ \sum_{i\geq1}e^{\beta\xi_{i}} = \infty\right\},
\]
it follows from the assumptions that $\mathbb P\left\{ \sum_{i\geq1}e^{\beta\xi_{i}} < \infty\right\}=1$, and the same with $\beta'$.

Conditioned on $\xi$, the process $\eta$ is
a sum of independent Poisson processes, and therefore $\eta$ is a Poisson point process
of finite intensity
\[
\sum_{i\geq1}e^{-\beta x+\beta\xi_{i}} dx=e^{-\beta x}\sum_{i\geq1}e^{\beta\xi_{i}} dx.
\]
Equivalently, with  $Z_{\beta}\triangleq\frac{1}{\beta}\log\left(\sum_{i\geq1}e^{\beta\xi_{i}}\right)$,
\begin{equation}
\eta\sim SDPPP\left(e^{-\beta x},\delta_{0},Z_\beta\right).\label{eq:etaSDPPP}
\end{equation}

Each $P_{i}$ satisfies (SDP) and thus (SUS) with the function ${\rm {Gum}}\left(\beta y\right)$
(cf. the converse part of Theorem 9). From Corollary \ref{cor:frz_specific_func}
\[
L_{P_{i}}\left[\left.e^{\beta'x}\,\right|\,\cdot\,\right]\approx{\rm {Gum}}\left(\beta\left(\,\cdot\,\right)\right).
\]
Hence,
\begin{align*}
\mathbb{E}\exp\left\{ -\left\langle \theta_{y}e^{\beta'x},\eta\right\rangle \right\}  & =\mathbb{E}\mathbb{E}\left[\left.\prod_{i\geq1}\exp\left\{ -\left\langle \theta_{y}e^{\beta'x},\theta_{\xi_{i}}P_{i}\right\rangle \right\} \right|\xi\right]\\
 & =\mathbb{E}\left[\prod_{i\geq1}{\rm {Gum}}\left(\beta\left(y-\xi_{i}-\tau_{1}\left(\beta,\beta'\right)\right)\right)\right]\\
 & =\mathbb{E}\left[\exp\left\{ -\sum_{i\geq1}e^{-\beta\left(y-\xi_{i}-\tau_{1}\left(\beta,\beta'\right)\right)}\right\} \right]\\
 & =\mathbb{E}\exp\left\{ -\left\langle \theta_{y-\tau_{1}\left(\beta,\beta'\right)}e^{\beta x},\xi\right\rangle \right\} =g\left(y-\tau_{1}\left(\beta,\beta'\right)-\tau_{2}\left(\beta\right)\right),
\end{align*}
for some $\tau_{1}\left(\beta,\beta'\right),\,\tau_{2}\left(\beta\right)\in\mathbb{R}$.

On the other hand, from (\ref{eq:etaSDPPP}),
\begin{align*}
\mathbb{E}\exp\left\{ -\left\langle \theta_{y}e^{\beta'x},\eta\right\rangle \right\}  & =\mathbb{E}\mathbb{E}\left[\left.\exp\left\{ -\left\langle \theta_{y}e^{\beta'x},\theta_{Z_{\beta}}P_{1}\right\rangle \right\} \right|\xi\right]\\
 & =\mathbb{E}\left[{\rm {Gum}}\left(\beta\left(y-Z_{\beta}-\tau_{1}\left(\beta,\beta'\right)\right)\right)\right].
\end{align*}

We conclude that, for any $\beta>c$,
\[
g\left(y\right)=\mathbb{E}\left[{\rm {Gum}}\left(\beta\left(y-Z_{\beta}+\tau_{2}\left(\beta\right)\right)\right)\right].
\]

Since the left hand side of the above inequality is constant, it
follows that $\left\{ Z_{\beta}+\tau_{2}\left(\beta\right)\right\} _{\beta>c}$
is a tight family. Therefore there exists a sequence $\beta_{i}$
that decreases to $c$ as $i\to\infty$ such that $Z_{\beta_{i}}+\tau_{2}\left(\beta_{i}\right)$
converges in distribution to some limiting random variable $Z$. For
this variable,
\[
g\left(y\right)=\lim_{i\to\infty}\mathbb{E}\left[{\rm {Gum}}\left(\beta_{i}\left(y-Z_{\beta_{i}}+\tau_{2}\left(\beta_{i}\right)\right)\right)\right]=\mathbb{E}\left[{\rm {Gum}}\left(c \left(y-Z\right)\right)\right],
\]
 which completes the proof.
\end{proof}
We remark that when $\xi'=\theta_S \xi$ with $S$ and $\xi$ independent, reassuringly, $$\frac{1}{\beta}\log\Big(\sum_{i\geq1}e^{\beta\xi^{\prime}_{i}}\Big) = S + \frac{1}{\beta}\log\Big(\sum_{i\geq1}e^{\beta\xi_{i}}\Big),$$ and the limiting procedure above `extracts' the shift.

\begin{rem}
  In Proposition \ref{prop:gum} of the appendix, we use the reasoning above
to
relate Gumbel distribution functions of different scales to each other by convolution with a Borel measure, a result of possible independent interest.
\end{rem}
We now complete the  proof of part $(c)$.
\begin{cor}
Part $(c)$ of the direct part of Theorem \ref{thm:main} holds.
\end{cor}
\begin{proof}
We wish to apply Lemma \ref{lem:150414-1} with $c=c_\xi$. The first condition (equivalence to $g$) is assumed directly. What remains is to show that $\sup_{y\in\mathbb R}g(y)=1$. This follows from (US) and Lemma \ref{lem:g monotone}.
\end{proof}
We continue with the proof of part $(d)$.
\begin{lem}
\label{lem:cLLN}
Part $(d)$ of the direct part of Theorem \ref{thm:main} holds if
$D\left(\xi\right)$ satisfies an LLN with rate $c_\xi$.
\end{lem}
\begin{proof}

Throughout the proof let $c=c_{\xi}$. Let $\psi=\psi\left(\xi\right)\sim DPPP\left(e^{-cx}dx,D\left(\xi\right)\right)$.
From Lemma \ref{lem:23} we know that $\xi\overset{d}{=}\theta_{Z'}\psi'$
for some random variable $Z'$ and that $\theta_{Z_{G}}\theta_{Z'}\psi'\overset{d}{=}\theta_{Z_{g}}\psi$,
with
\begin{equation}
Z_{G}\perp\left(Z',\psi'\right)\mbox{\,\,\, and\,\,\,}Z_{g}\perp\psi,\label{eq:indep}
\end{equation}
where $\psi'\overset{d}{=}\psi$ and where $Z_{G}$ and $Z_{g}$ are
random variables with cumulative distribution functions ${\rm {Gum}}^{\left(c\right)}$
and $g$, respectively.

From the same argument as in the proof of Lemma \ref{lem:23}, appealing
to Corollary 6.11 of \cite{Kallenberg2nd} we can define all of the
variables on the same probability space such that (\ref{eq:indep})
is preserved and such that
\begin{equation}
\theta_{Z_{G}}\theta_{Z'}\psi'=\theta_{Z_{g}}\psi\,,\,\,\,\mbox{almost surely}.\label{eq:coupling1}
\end{equation}

Since $D(\xi)$ satisfies the LLN, there exists a sequence $y_n\in \mathbb R$ that increases to $\infty$, such that
$\psi\left(\left(-y_n,\infty\right)\right)/\alpha(y_n)\to u$ almost surely, as $n\to\infty$, for some $u>0$. Note that, almost surely,
\begin{align*}
\frac{\psi\left(\left(-y_n,\infty\right)\right)}{\alpha(y_n)}&=\frac{\theta_{Z_g}\psi\left(\left(-y_n+Z_g,\infty\right)\right)}{\alpha(y_n)} \\
&=\frac{\theta_{Z_G+Z'}\psi'\left(\left(-y_n+Z_g,\infty\right)\right)}{\alpha(y_n)}=\frac{\psi'\left(\left(-y_n+Z_g-Z_G-Z',\infty\right)\right)}{\alpha(y_n)}\, ,
\end{align*}
and therefore the same convergence holds for the rightmost term of the equation.

We shall prove by contradiction that $Z_g-Z_G-Z'\geq 0$ almost surely. Assume otherwise, i.e., there exists $\epsilon >0$ such that
\[
\mathbb P\{Z_g-Z_G-Z'>\epsilon\}>0\,.
\]
Since $\psi\left(\left(-y,\infty\right)\right)$ is increasing in $y$, by \eqref{eq:alpharegvar}, almost surely on the event $\{Z_g-Z_G-Z'>\epsilon\}$,
\[
\liminf_{n\to\infty} \frac{\psi\left(\left(-(y_n-\epsilon),\infty\right)\right)}{\alpha(y_n-\epsilon)}\geq\lim_{n\to\infty}\frac{\psi\left(\left(-y_n+Z_g-Z_G-Z',\infty\right)\right)}{\alpha(y_n-\epsilon)}=ue^{c\epsilon}\, ,
\]
and therefore, for large enough $N$,
\[
\mathbb P\left\{ \forall n\geq N:\ \frac{\psi\left(\left(-(y_n-\epsilon),\infty\right)\right)}{\alpha(y_n-\epsilon)} \geq ue^{c\epsilon}/2 \right\}>0\, ,
\]
in contradiction to the LLN. Hence $Z_g-Z_G-Z'\geq 0$ almost surely.

One obtains similarly that $Z_g-Z_G-Z'\leq 0$
and thus $Z_g=Z_G+Z'$, almost surely. Since $Z_G\perp Z'$, this implies
that $Z_g$ is distributed like a shifted Gumbel, i.e. Condition (SUS) holds.
The proof is completed.
\end{proof}
We finish the proof of the direct part of Theorem \ref{thm:main}
with the following.
\begin{lem}
Assuming $D\left(\xi\right)$ has exponential
moments, part $(d)$ of the direct part of Theorem \ref{thm:main}
holds.\end{lem}
\begin{proof}
As in the proof of Lemma \ref{lem:cLLN}, abbreviate $c=c_{\xi}$ and
couple the shifts $Z_G$, $Z_g$, $Z'$ and point processes $\psi$, $\psi'$ defined there, so that \eqref{eq:indep} and \eqref{eq:coupling1} hold.
%

For any point process $\eta$ such that $\mathcal{M}\left(\eta\right)<\infty$
almost surely, we define $\pi\left(\eta\right)\triangleq\theta_{-\mathcal{M}\left(\eta\right)}\eta$.
That is, $\pi\left(\eta\right)$ is $\eta$ shifted so its rightmost
particle is exactly at zero almost surely. Clearly,
\[
\pi\left(\theta_{Z_{G}}\theta_{Z'}\psi'\right)=\pi\left(\psi'\right)\mbox{\,\,\ and\,\,}\pi\left(\theta_{Z_{g}}\psi\right)=\pi\left(\psi\right),
\]
and thus $\pi\left(\psi'\right)=\pi\left(\psi\right)$ almost surely.

Now, suppose $A\subset\mathcal{N}$ is a measurable set such that
$\left\{ \pi\left(\psi\right)\in A\right\} $ has positive probability.
Note that, since $\pi\left(\psi'\right)=\pi\left(\psi\right)$ almost
surely, the probability of the symmetric difference $\left\{ \pi\left(\psi\right)\in A\right\} \bigtriangleup\left\{ \pi\left(\psi'\right)\in A\right\} $
is $0$. Therefore, conditioned either on $\left\{ \pi\left(\psi\right)\in A\right\} $
or on $\left\{ \pi\left(\psi'\right)\in A\right\} $, the random vector
$\left(Z_{G},Z',\mathcal{M}\left(\psi'\right)\right)$ has the same
distribution.

Define $X_{A}$ and $Y_{A}$ to be random variables distributed as
$Z'+\mathcal{M}\left(\psi'\right)$ and $\mathcal{M}\left(\psi\right)$
conditioned on $\left\{ \pi\left(\psi\right)\in A\right\} $, respectively.
Then from (\ref{eq:indep}), by considering the rightmost particle
of the point processes in (\ref{eq:coupling1}),
\[
Z_{G}+X_{A}\overset{d}{=}Z_{g}+Y_{A},\,\,\mbox{where}\,\, Z_{G}\perp X_{A}\,\,\mbox{and}\,\, Z_{g}\perp Y_{A}.
\]

Let $A_{n}\subset\mathcal{N}$, $n\geq1$, be a sequence of measurable
sets such that $\left\{ \pi\left(\psi\right)\in A_{n}\right\} $ has
positive probability and abbreviate $Y_{n}=Y_{A_{n}}$, $X_{n}=X_{A_{n}}$.
Suppose that $Y_{n}-c_{n}\overset{d}{\rightarrow}0$ as $n\rightarrow\infty$
for some deterministic sequence $c_{n}\in\mathbb{R}$. Then $Z_{G}+X_{n}-c_{n}\overset{d}{\rightarrow}Z_{g}$
as $n\rightarrow\infty$ (with $Z_{G}\perp X_{n}-c_{n}$ for any $n\geq1$),
thus the characteristic function of $Z_{g}$ satisfies
\[
\mathbb{E}e^{-itZ_{g}}=\lim_{n\rightarrow\infty}\mathbb{E}e^{-it\left(Z_{G}+X_{n}-c_{n}\right)}=\mathbb{E}e^{-itZ_{G}}\lim_{n\rightarrow\infty}\mathbb{E}e^{-it\left(X_{n}-c_{n}\right)}.
\]

It is easily seen from the convergence of $Z_{G}+X_{n}-c_{n}$ and
the independence of $Z_{G}$ on $X_{n}-c_{n}$ that $X_{n}-c_{n}$,
$n\geq1$, is tight. By the continuity theorem (cf. Theorem 2, p.
481 of \cite{Feller2nd}), this implies that the limit of $\mathbb{E}e^{-it\left(X_{n}-c_{n}\right)}$
is the characteristic function of some random variable $Z$ and thus
$Z_{g}=Z_{G}+Z$ with $Z_{G}\perp Z$.

This is exactly what we need to show, hence all that remains is to
construct the sets $A_{n}$ as above. We shall show that for
\[
A_{n}^{\epsilon}\triangleq\left\{ \eta\in\mathcal{N}:\,\mathcal{M}\left(\eta\right)=0,\,\eta\left(\left[-\epsilon,0\right]\geq n\right)\right\} ,
\]
with $\epsilon>0$ being a parameter to be determined below, the convergence
$Y_{n}-c_{n}\overset{d}{\rightarrow}0$ is achieved. Of course, $\mathbb{P}\left\{ \pi\left(\psi\right)\in A_{n}^{\epsilon}\right\} >0$
for any $\epsilon>0$, $n\geq1$, since with positive probability
there are $n$ atoms of the Poisson process corresponding to $\psi$
in the interval $\left(\mathcal{M}\left(\xi\right)-\epsilon,\mathcal{M}\left(\xi\right)\right)$.

For simplicity we shall assume $c=1$, the general case follows by scaling.
Recall that the decoration $D\left(\xi\right)$ satisfies $\mathcal{M}\left(D\left(\xi\right)\right)=0$,
almost surely. Therefore,
the density of $\mathcal{M}\left(\psi\right)$ relative
to Lebesgue measure is given by $\exp\left\{ -e^{-x}-x\right\} $.

By conditioning on $\mathcal{M}\left(\psi\right)$, for any Borel
set $B\subset\mathbb{R}$,
\begin{align}
\mu_{n}\left(B\right) & \triangleq\mathbb{P}\left\{ \mathcal{M}\left(\psi\right)\in B\,\left|\,\pi\left(\psi\right)\in A_{n}^{\epsilon}\right.\right\} \nonumber \\
 & =\frac{\int_{B}\exp\left\{ -e^{-x}-x\right\} \mathbb{P}\left\{ \left(\eta_{x}+\theta_{x}D_{0}\right)\left(\left[x-\epsilon,x\right]\right)\geq n\right\} dx}{\int_{\mathbb{R}}\exp\left\{ -e^{-x}-x\right\} \mathbb{P}\left\{ \left(\eta_{x}+\theta_{x}D_{0}\right)\left(\left[x-\epsilon,x\right]\right)\geq n\right\} dx},\label{eq:cndprb}
\end{align}
where $\eta_{x}\sim DPPP\left(e^{-cy}\cdot\mathbf{1}_{\left(-\infty,x\right)}dy,D\left(\xi\right)\right)$
and $D_{0}$ is an independent copy of $D\left(\xi\right)$.

The number of the decorations composing $\eta_{x}$ that attribute
a positive measure to the interval $\left[x-\epsilon,x\right]$ is
a Poisson variable of parameter $e^{-x}\left(e^{\epsilon}-1\right)$.
Conditioned on this number, the decorations are independent, and each
is equal in distribution to $\theta_{\bar{S}\left(\epsilon\right)}D\left(\xi\right)$
where the p.d.f.  of the random variable
$\bar{S}\left(\epsilon\right)$ is
$$e^{-s}\mathbf{1}_{\left[x-\epsilon,x\right]}\left(s\right)ds/\left(e^{-x}\left(e^{\epsilon}-1\right)\right).$$
Note also
that $\bar{S}\left(\epsilon\right)\overset{d}{=}x+S\left(\epsilon\right)$
where the p.d.f. of $S\left(\epsilon\right)$ is $e^{-s}\mathbf{1}_{\left[-\epsilon,0\right]}\left(s\right)ds/\left(\left(e^{\epsilon}-1\right)\right)$.
Hence, if
\begin{align*}
U_{0}^{\epsilon} & \overset{d}{=}\theta_{x}D_{0}\left(\left[x-\epsilon,x\right]\right)=D_{0}\left(\left[-\epsilon,0\right]\right),\\
W_{i}^{\epsilon} & \overset{d}{=}\theta_{\bar{S}\left(\epsilon\right)}D\left(\xi\right)\left(\left[x-\epsilon,x\right]\right)=\theta_{S\left(\epsilon\right)}D\left(\xi\right)\left(\left[-\epsilon,0\right]\right),\,\,\,\forall i\geq1,
\end{align*}
are independent random variables, with $\left\{ W_{i}^{\epsilon}\right\} _{i\geq1}$
an i.i.d sequence, then
\[
\left(\eta_{x}+\theta_{x}D_{0}\right)\left(\left[x-\epsilon,x\right]\right)\overset{d}{=}U_{0}^{\epsilon}+\sum_{i=1}^{\bar{N}_{x}}W_{i}^{\epsilon},
\]
where $\bar{N}_{x}^{\epsilon}\sim\mbox{Pois}\left(e^{-x}\left(e^{\epsilon}-1\right)\right)$
is a random variable independent of $U_{0}$, $W_{i}$, $i\geq1$.

From this, by a change of variables $y=e^{-x}/n$, we obtain from
(\ref{eq:cndprb}), for any $y_{n}^{*}$, $\delta>0$,
\begin{equation}
\mu_{n}\left(\left[-\log\left(ny_{n}^{*}\right)-\delta,-\log\left(ny_{n}^{*}\right)+\delta\right]\right)=\frac{\int_{e^{-\delta}y_{n}^{*}}^{e^{\delta}y_{n}^{*}}h\left(y,n\right)dy}{\int_{0}^{\infty}h\left(y,n\right)dy},\label{eq:33}
\end{equation}
where
\[
h\left(y,n\right)\triangleq e^{-ny}\mathbb{P}\left\{ U_{0}^{\epsilon}+\sum_{i=1}^{N_{ny}}W_{i}^{\epsilon}\geq n\right\} =e^{-ny}\mathbb{P}\left\{ \frac{1}{n}U_{0}^{\epsilon}+\frac{1}{n}\sum_{i=1}^{N_{ny}}W_{i}^{\epsilon}\geq1\right\}
\]
and where $N_{ny}^{\epsilon}\sim\mbox{Pois}\left(ny\left(e^{\epsilon}-1\right)\right)$
is a random variable independent of $U_{0}^{\epsilon}$, $W_{i}^{\epsilon}$,
$i\geq1$.

In order to prove the convergence of $Y_{n}-c_{n}\overset{d}{\rightarrow}0$
as $n\rightarrow\infty$, it suffices to show that there exist
$y_{n}^{*}\in\left(0,\infty\right)$, $n\geq1$, such that the ratio
of (\ref{eq:33}) converges to $1$ as $n\rightarrow\infty$, for
any $\delta>0$; one then sets $c_{n}=-\log\left(ny_{n}^{*}\right)$.
In our proof, we will choose $y_n=m^*$ where $m^*$ is an $n$-independent
constant to be determined below.

We first choose the
parameter $\epsilon>0$. Recall that, by assumption,
there exists $\epsilon_{0}>0$ such that for any $\epsilon\in\left[0,\epsilon_{0}\right]$,
there exists $t>0$ for which $\mathbb{E}\left[\exp\left\{ tD\left(\xi\right)\left(-\epsilon,0\right)\right\} \right]<\infty$.
For each such $\epsilon$, we define
\[
\lambda^{*}\left(U_{0}^{\epsilon}\right)\triangleq\inf\left\{ \lambda:\,\mathbb{E}\left\{ \exp\left(\lambda U_{0}^{\epsilon}\right)\right\} =\infty\right\} ,
\]
with $\lambda^{*}\left(U_{0}^{\epsilon}\right)=\infty$ in case
the set is empty. Define $\lambda^{*}\left(W_{1}^{\epsilon}\right)$
similarly and note that since $0\leq W_{1}^{\epsilon}\leq U_{0}^{\epsilon}$
a.s., $\lambda^{*}\left(U_{0}^{\epsilon}\right)\leq\lambda^{*}\left(W_{1}^{\epsilon}\right)$.

If $\lambda^{*}\left(U_{0}^{\epsilon_{0}}\right)=\infty$, and therefore
also $\lambda^{*}\left(W_{1}^{\epsilon_{0}}\right)=\infty$, we set
$\epsilon=\epsilon_{0}$. Otherwise, note that $\lambda^{*}\left(U_{0}^{\epsilon}\right)$
is a bounded, decreasing function of $\epsilon$ on $\left[0,\epsilon_{0}\right]$,
thus there exists some $\epsilon'>0$ at which it is continuous, and
we set $\epsilon=\epsilon'$. For any $\lambda>\lambda^{*}\left(U_{0}^{\epsilon}\right)$,
\begin{align*}
\mathbb{E}\left\{ \exp\left(\lambda W_{1}^{\epsilon}\right)\right\}  & \geq\mathbb{P}\left\{ S\left(\epsilon\right)\geq-\delta\right\} \mathbb{E}\left\{ \exp\left(\lambda\cdot\theta_{-\delta}D\left(\xi\right)\left(\left[-\epsilon,0\right]\right)\right)\right\} \\
 & =\mathbb{P}\left\{ S\left(\epsilon\right)\geq-\delta\right\} \mathbb{E}\left\{ \exp\left(\lambda\cdot U_{0}^{\epsilon-\delta}\right)\right\} .
\end{align*}
Choosing $\delta>0$ small enough such that $\lambda>\lambda^{*}\left(U_{0}^{\epsilon-\delta}\right)$,
we conclude that $\mathbb{E}\left\{ \exp\left(\lambda W_{1}^{\epsilon}\right)\right\} =\infty$.
It follows that $\lambda^{*}\left(W_{1}^{\epsilon}\right)=\lambda^{*}\left(U_{0}^{\epsilon}\right)$.
Henceforth we fix $\epsilon$ as above,
abbreviate $\lambda^{*}\triangleq\lambda^{*}\left(U_{0}^{\epsilon}\right)$,
and suppress $\epsilon$ from the notation.

For each $y>0$, define
the logarithmic
moment generating function
\begin{equation}
  \label{eq:lambda_y}
  \Lambda_{y}\left(\lambda\right)\triangleq\log\mathbb{E}\left\{ \exp\left\{ \lambda\sum_{i=1}^{N_{y}}W_{i}\right\} \right\}
  =y\left(e^{\epsilon}-1\right)\left(\mathbb{E}\left\{
    e{}^{\lambda W_{1}}\right\} -1\right)\,, (\lambda\in \mathbb{R})
  \end{equation}
    and its Fenchel-Legendre transform
    \[
  \Lambda_{y}^{*}\left(x\right)\triangleq\sup_{\lambda\in\mathbb{R}}\left\{ \lambda x-\Lambda_{y}\left(\lambda\right)\right\}\,, (x\in \mathbb{R}) .
\]
We next claim that
\begin{equation}
  \label{eq-prehLDP}
 \lim_{n\rightarrow\infty}\frac{1}{n}\log\mathbb{P}\left\{ \frac{1}{n}U_{0}+
 \frac{1}{n}\sum_{i=1}^{N_{ny}}W_{i}\geq1\right\}=
 -\inf_{x\geq1}\Lambda_{y}^{*}\left(x\right).
\end{equation}
Indeed, introduce the representation
\begin{equation}
  \label{eq-rep}
\sum_{i=1}^{N_{ny}}W_{i}\overset{d}{=}\sum_{j=1}^{n}\sum_{i=1}^{N_{y}^{\left(j\right)}}W_{i}^{\left(j\right)}
\end{equation}
with the i.i.d. sequences of variables (in $i,j$)
$W_{i}^{\left(j\right)}\overset{d}{=}W_{1}$ and
$N_{y}^{\left(j\right)}\overset{d}{=}N_{y}$ independent of each other.
The lower bound in \eqref{eq-prehLDP} follows from  $U_0\geq 0$
using Cram\'{e}r's theorem
\cite[Corollary 2.2.19]{LDbook}:
\begin{align*}
 & \negthickspace\negthickspace\negthickspace\negthickspace
 \liminf_{n\rightarrow\infty}\frac{1}{n}\log\mathbb{P}\left\{ \frac{1}{n}U_{0}+\frac{1}{n}\sum_{i=1}^{N_{ny}}W_{i}\geq1\right\} \\
 & \geq\liminf_{n\rightarrow\infty}\frac{1}{n}\log\mathbb{P}\left\{ \frac{1}{n}\sum_{j=1}^{n}\sum_{i=1}^{N_{y}^{\left(j\right)}}W_{i}^{\left(j\right)}\geq1\right\} =-\inf_{x\geq1}\Lambda_{y}^{*}\left(x\right).
\end{align*}
To see the corresponding upper bound, introduce the logarithmic
%
moment generating function
\[
  \bar{\Lambda}_{n}^{\left(y\right)}\left(\lambda\right)\triangleq\log\mathbb{E}\left\{ \exp\left\{ \lambda\left(U_{0}/n+\sum_{i=1}^{N_{ny}}W_{i}/n\right)\right\} \right\} \,\,\,\, (\lambda\in \mathbb{R}),
\]
and its scaled limit and Fenchel-Legendre transform
\[
  \bar{\Lambda}_{y}\left(\lambda\right)\triangleq\lim_{n\rightarrow\infty}\frac{1}{n}\bar{\Lambda}_{n}^{\left(y\right)}\left(n\lambda\right)\,,
  \,\,\,\,
  \bar{\Lambda}_{y}^{*}\left(x\right)\triangleq\sup_{\lambda\in\mathbb{R}}\left\{ \lambda x-\bar{\Lambda}_{y}\left(\lambda\right)\right\} \,\,\,\, (x\in \mathbb{R}).
  \]
  For $\lambda<\lambda^*$ we have (using the representation \eqref{eq-rep}) that
  $\Lambda_y(\lambda)=\bar{\Lambda}_y(\lambda)$. Obviously, for
  $\lambda>\lambda^*$ we have
  $\Lambda_y(\lambda)=\bar{\Lambda}_y(\lambda)=\infty$.
  Thus, $\Lambda_y$ and $\bar{\Lambda}_y$ coincide in the
  interior of their (common) domain;
  since both are convex, they are necessarily
  continuous in the interior of their domain and
  $$\limsup_{\lambda\nearrow\lambda^*} \Lambda_y(\lambda)\leq \Lambda_y(\lambda^*),\;
  \limsup_{\lambda\nearrow\lambda^*} \bar{\Lambda}_y(\lambda)\leq \bar{\Lambda}_y
  (\lambda^*).$$
In particular,
for any $x\in\mathbb{R}$,
\begin{align*}
\bar{\Lambda}_{y}^{*}\left(x\right) & =\sup_{\lambda\in\mathbb{R}}\left\{ \lambda x-\bar{\Lambda}_{y}\left(\lambda\right)\right\} =\sup_{\lambda<\lambda^{*}}\left\{ \lambda x-\bar{\Lambda}_{y}\left(\lambda\right)\right\} \\
 & =\sup_{\lambda<\lambda^{*}}\left\{ \lambda x-\Lambda_{y}\left(\lambda\right)\right\} =\sup_{\lambda\in\mathbb{R}}\left\{ \lambda x-\Lambda_{y}\left(\lambda\right)\right\} =\Lambda_{y}^{*}\left(x\right).
\end{align*}

Applying the upper bound of
the G\"{a}rtner\textendash{}Ellis theorem  \cite[Theorem 2.3.6]{LDbook},
we obtain
\[
\limsup_{n\rightarrow\infty}\frac{1}{n}\log\mathbb{P}\left\{ U_{0}/n+\sum_{i=1}^{N_{ny}}W_{i}/n\geq1\right\} \leq-\inf_{x\geq1}\Lambda_{y}^{*}\left(x\right),
\]
which completes the proof of \eqref{eq-prehLDP}.

Note that
with  $\bar{x}\triangleq\left(e^{\epsilon}-1\right)\mathbb{E}W_{1}$,
we have  by \cite[Lemma 2.2.5]{LDbook} that
$\Lambda_1^*(\bar x)=0$ and that $\Lambda_1^*(x)$ is nondecreasing for $x\geq \bar x$. Thus, we can
rewrite \eqref{eq-prehLDP} as
\begin{eqnarray}
  \label{eq-hLDP}
  \nonumber
\hat{h}\left(y\right)&\triangleq&
\lim_{n\rightarrow\infty}\frac{1}{n}\log h\left(y,n\right)=-y-\inf_{x\geq1}\Lambda_{y}^{*}\left(x\right)=-y-\inf_{x\geq1/y}\Lambda_{1}^{*}\left(x\right)\\
&=&
\begin{cases}
-y-\Lambda_{1}^{*}\left(1/y\right) & \mbox{, if }y\in\left(0,1/\bar{x}\right)\\
-y & \mbox{, if }y\in\left[1/\bar{x},\infty\right),
\end{cases}
\end{eqnarray}
where the next-to-last equality follows from (\ref{eq:lambda_y}) and the
definition of the Fenchel-Legendre transform.

Before returning to the evaluation of \eqref{eq:33}, we analyze the function
$\hat h$. As a preliminary, note that $\Lambda_1^*$ is finite and continuous
 in a neighborhood of $\bar x$ because $0$ is in $\mathcal D_{\Lambda_1}^o$, the interior of
the domain of $\Lambda_1$, and $\sum_{i=1}^{N_y}W_i$ is not deterministic, and thus $\bar x \in \{ {\Lambda_1} '(\lambda): \, \lambda \in \mathcal D_{\Lambda_1}^o \}$, so
that \cite[Lemma 2.2.5(c) and Exercise 2.2.24]{LDbook} can be applied.

Let $I_1=(0,1/\bar x)$ and $I_2=[1/\bar x,\infty)$. We  have that $\hat{h}$
	is strictly decreasing on $I_2$. Further, since $1/y$ is strictly convex in $y$, $\hat{h}(y)$ is
	strictly concave in $1/y$ on $I_1$. Thus, since $\lim_{y\searrow 0} \hat{h}(y)=-\infty$,
	$$\sup_{y\in \mathbb{R}} \hat{h}(y)=
	\sup_{z\in [\bar x,\infty)}\hat{h}(1/z)=\hat{h}(m^*)$$
for a unique $m^*\in (0,1/\bar x]$, and, fixing an arbitrary $\delta > 0$,  with $J=(e^{-\delta}m^*,e^{\delta}
m^*)$
we have
	\[
\rho\triangleq
\hat{h}\left(m^*\right)>\kappa\triangleq\sup_{y\in J^{c}}\hat{h}\left(y\right)\,.
\]
Set $\Delta\triangleq\rho-\kappa.$

Let $\mathcal{P}=\left\{ 0=p_{0}<p_{1}<\cdots<p_{k}=T\right\} $ be
a finite partition containing each of the ends of the interval $J$,
such that $\left\Vert \mathcal{P}\right\Vert \triangleq\max_{i\leq k}\left(p_{i}-p_{i-1}\right)<\Delta/6$,
and where $T$ satisfies
$-T<\kappa+\Delta/3$. Define for any point $y\in\left[0,T\right]$,
\[
p^{*}\left(y\right)=\min\left\{ p_{i}:\, p_{i}\geq y\right\} \,\,\mbox{and}\,\, p_{*}\left(y\right)=\max\left\{ p_{i}:\, p_{i}\leq y\right\} .
\]

Set $J_T^c=J^c\cap [0,T]$.
Since $Q_{y}\triangleq\mathbb{P}
\left\{ \frac{1}{n}U_{0}^{\epsilon}+
\frac{1}{n}\sum_{i=1}^{N_{ny}}W_{i}^{\epsilon}\geq1\right\} $
is nondecreasing in $y$, we have the upper bound
\begin{align*}
 & \negthickspace\negthickspace\negthickspace\negthickspace\int_{J_T^c}h\left(y,n\right)dy\\
 & \leq\int_{J_T^c}e^{-np_{*}\left(y\right)}Q_{p^{*}\left(y\right)}dy\leq
 T\sup_{y\in J_T^c}\left\{
	 e^{-np_{*}\left(y\right)}Q_{p^{*}\left(y\right)}\right\} \\
 & \leq T e^{n\left\Vert \mathcal{P}\right\Vert }
 \max_{y\in\left(J_T^c\right)\cap\mathcal{P}}
 \left\{ e^{-ny}Q_{y}\right\} \leq T e^{n\Delta/6}
 \max_{y\in\left(J_T^c\right)\cap\mathcal{P}}h\left(y,n\right),
\end{align*}
and, for any $\alpha < (e^\delta -1)m^*$, the lower bound
$$\int_{J}h\left(y,n\right)dy \geq \int_{m^*}^{m^*+\alpha}h\left(y,n\right)dy \geq \alpha h\left(m^*,n\right) e^{-n\alpha}\,.$$
Therefore, for $n$ large enough,
using the fact that $\mathcal{P}$ is a finite set which does not depend on $n$,
\begin{align*}
\frac{1}{n}\log\left(\int_{J_T^c}h\left(y,n\right)dy\right) & < \kappa+\Delta/3,\\
\frac{1}{n}\log\left(\int_{J}h\left(y,n\right)dy\right) & > \rho-\Delta/3.
\end{align*}

For bounding the integral on $\left(T,\infty\right)$ we use the simple
bound
\[
\frac{1}{n}\log\left(\int_{T}^{\infty}h\left(y,n\right)dy\right)\leq\frac{1}{n}\log\left(\int_{T}^{\infty}e^{-ny}dy\right)\leq -T<\kappa+\Delta/3.
\]

Combining the three bounds we obtain that
\[
\lim_{n\rightarrow\infty}\frac{\int_{J}h\left(y,n\right)dy}{\int_{J^c}h\left(y,n\right)dy}=\infty.
\]
Since $\delta>0$ was arbitrary, with our choice $y_{n}^{*}=m^{*}$,
the ratio of (\ref{eq:33}) converges as $n\rightarrow\infty$ to
$1$, which completes the proof.
\end{proof}

\section{\label{sec:Relation-to-Freezing}From uniqueness to freezing}

In our notation, the freezing phenomenon considered in the physics
literature \cite{CLD,DerSpo,Fyodorov,F&B1,Fyo3,Fyo2,Fyo1} can be
written as
\[
\forall\beta>\beta_{c}:\,\,\lim_{n\to\infty}L_{\xi_{n}}\left[\left.e^{\beta x}\,\right|\,\cdot\,\right]\thickapprox g\left(\cdot\right),\mbox{ where }\xi_{n}\triangleq\sum_{i\leq N\left(n\right)}\delta_{X_{i}\left(n\right)-m_{n}}.
\]
Under the convergence $\xi_{n}\overset{d}{\to}\xi$,
as $n\to\infty$, it is natural to examine
the relations of the condition above with
\[
\forall\beta>\beta_{c}:\,\, L_{\xi}\left[\left.e^{\beta x}\,\right|\,\cdot\,\right]\thickapprox g\left(\cdot\right).
\]

Essentially, if enough is known about a point process $\xi$ for which
$L_{\xi}\left[\left.f\,\right|\,\cdot\,\right]$ is uniquely supported
(i.e., the above holds for functions in $C_{c}^{+}\left(\mathbb{R}\right)$),
one can extend the equivalence $L_{\xi}\left[\left.f\,\right|\,\cdot\,\right]\thickapprox g\left(\cdot\right)$
to functions with unbounded support, e.g. $e^{\beta x}$, by appealing
to Corollary \ref{cor:frz_specific_func} and approximating functions by sequences
in $C_{c}^{+}\left(\mathbb{R}\right)$.

However, since convergence of point processes in distribution (with
respect to the vague topology) depends only on the convergence of
$\left\langle f,\xi_{n}\right\rangle $ for $f$ with bounded support,
some care is needed in concluding that
\[
\lim_{n\to\infty}L_{\xi_{n}}\left[\left.e^{\beta x}\,\right|\,\cdot\,\right]=L_{\xi}\left[\left.e^{\beta x}\,\right|\,\cdot\,\right].
\]
The purpose of this section is to give a sufficient condition for
this in terms of the sequence $\xi_{n}$.

We begin with a remark on the assumption of tightness. Assume that
$\xi_{n}$, $n\geq1$, is a sequence of point processes such that
$\lim_{n\rightarrow\infty}L_{\xi_{n}}\left[\left.f\,\right|\, y\right]\triangleq\varphi\left(\left.f\,\right|\, y\right)$
exists for any $f\in C_{c}^{+}\left(\mathbb{R}\right)$, $\varphi\left(\left.f\,\right|\,\cdot\right)\thickapprox g\left(\cdot\right)$.
Since the limit of the Laplace functionals exists, the sequence converges
in distribution if and only if it is tight. One then may hope that
since we are not only assuming convergence of the Laplace functionals,
but also assume this limit to have a very specific form, tightness
should follow. The next example shows it is not so.
\begin{example}
Let $\xi_{n}=\sum_{x\in n^{-1}\mathbb{Z}^{-}}\delta_{x}$, where $\mathbb{Z}^{-}\triangleq\mathbb{Z}\cap\left(-\infty,0\right]$.
Then for any $f\in C_{c}^{+}\left(\mathbb{R}\right)$, $\lim_{n\rightarrow\infty}L_{\xi_{n}}\left[\left.f\,\right|\, y\right]=\mathbf{1}_{\left[0,\infty\right)}\left(y+\tau_{f}\right)$
with $\tau_{f}=\inf\left\{ x:f\left(x\right)>0\right\} $.

Let $g:\mathbb{\mathbb{R}\rightarrow R}$ be a distribution function
and let $Z$ be a random variable whose distribution function is $g$,
independent of $\xi_{n}$. Defining $\xi_{n}^{\prime}=\theta_{Z}\xi_{n}$
yields
\[
\lim_{n\rightarrow\infty}L_{\xi_{n}^{\prime}}\left[\left.f\,\right|\, y\right]=\mathbb{E}\left\{ \mathbf{1}_{\left[0,\infty\right)}\left(y-Z+\tau_{f}\right)\right\} =\mathbb{P}\left\{ Z\leq y+\tau_{f}\right\} =g\left(y+\tau_{f}\right).
\]

Defining $\xi_{n}^{\prime\prime}=A\xi_{n}^{\prime}$ where $A$ is
a random variable independent of $\xi_{n}^{\prime}$ that is equal
to $0$ with probability $p$ and is equal to $1$ with probability
$1-p$ , we obtain a sequence such that $\lim_{n\rightarrow\infty}L_{\xi_{n}^{\prime\prime}}\left[\left.f\,\right|\, y\right]\approx g_{p}\left(y\right)=p+\left(1-p\right)g\left(y\right)$.
Similarly, by reflecting the point processes around zero we obtain
a sequence of point processes such that $\lim_{n\rightarrow\infty}L_{\xi_{n}^{\prime\prime}}\left[\left.f\,\right|\, y\right]\approx-g_{p}\left(y\right)$.

In particular, for any function $g$ that satisfy the conditions stated
in Lemma \ref{lem:g monotone} we can construct a sequence of point
processes $\xi_{n}^{\prime\prime}$ which is not tight and such that
$\lim_{n\rightarrow\infty}L_{\xi_{n}^{\prime\prime}}\left[\left.f\,\right|\, y\right]\approx g\left(y\right)$
for any $f\in C_{c}^{+}\left(\mathbb{R}\right)$.\qed
\end{example}
In light of the above, we assume in the following convergence of the sequence of processes and not just the convergence of the Laplace functionals.
\begin{lem}
\label{lem:str_frz}Let $\xi_{n}$, $n\geq1$, be a sequence of point
processes such that $\xi_{n}\overset{d}{\rightarrow}\xi$ where $L_{\xi}\left[\left.f\,\right|\,\cdot\,\right]$
is uniquely supported on $\left[g\right]$. Then any nonnegative,
continuous function $f\neq0$ such that
\begin{equation}
\forall y\in\mathbb{R},\,\,\,\lim_{T\rightarrow\infty}\limsup_{n\rightarrow\infty}\mathbb{P}\left\{ \int\theta_{y}f\cdot\mathbf{1}_{\left(-T,T\right)^{c}}d\xi_{n}>\epsilon\right\} =0,\label{eq:1}
\end{equation}
satisfies
\[
\lim_{n\rightarrow\infty}L_{\xi_{n}}\left[\left.f\,\right|\,\cdot\,\right]=L_{\xi}\left[\left.f\,\right|\,\cdot\,\right]\approx g\left(\cdot\right).
\]
\end{lem}
\begin{proof}
For any $T>0$ let $h_{T}:\mathbb{R}\rightarrow\left[0,1\right]$
be the continuous function that is equal to $1$ on $\left[-T,T\right]$,
is equal to $0$ on $\left[-\left(T+1\right),T+1\right]^{c}$, and
is linear in each of the two remaining intervals. Denote $h_{T}^{c}\triangleq1-h_{T}$.
Let $f\neq0$ be a nonnegative, continuous function such that (\ref{eq:4})
holds and let $\epsilon>0$.

First, note that for any $T>0$ and $y\in\mathbb{R}$, abbreviating
$f_{y}\triangleq\theta_{y}f$,
\begin{align}
 & \negthickspace\negthickspace\negthickspace\negthickspace\mathbb{P}\left\{ \int f_{y}\cdot h_{T}^{c}d\xi>\epsilon\right\} =\lim_{R\rightarrow\infty}\mathbb{P}\left\{ \int f_{y}\cdot\left(h_{T}^{c}-h_{R}^{c}\right)d\xi>\epsilon\right\} \nonumber \\
 & \leq\lim_{R\rightarrow\infty}\liminf_{n\rightarrow\infty}\mathbb{P}\left\{ \int f_{y}\cdot\left(h_{T}^{c}-h_{R}^{c}\right)d\xi_{n}>\epsilon\right\} \nonumber \\
 & \leq\liminf_{n\rightarrow\infty}\mathbb{P}\left\{ \int f_{y}\cdot h_{T}^{c}d\xi_{n}>\epsilon\right\} ,\label{eq:2}
\end{align}
where the first inequality follows from the convergence in distribution
of $\int f_{y}\cdot\left(h_{T}^{c}-h_{R}^{c}\right)d\xi_{n}\rightarrow\int f_{y}\cdot\left(h_{T}^{c}-h_{R}^{c}\right)d\xi$
and the portmanteau theorem. Hence, since
\[
\mathbb{P}\left\{ \int f_{y}d\xi=\infty\right\} =\mathbb{P}\left\{ \int f_{y}\cdot h_{T}^{c}d\xi=\infty\right\} ,
\]
by (\ref{eq:1}), $\int f_{y}d\xi<\infty$ almost surely.

Let $T>0$. Embed $\xi$, $\xi_{n}$, $n\geq1$, in the same probability
space such that $\int f_{y}\cdot h_{T}d\xi_{n}\rightarrow\int f_{y}\cdot h_{T}d\xi$
in probability (for example, by using Skorohod coupling, cf. Corollary
6.12 of \cite{Kallenberg2nd}). Write
\begin{align*}
 & \negthickspace\negthickspace\negthickspace\negthickspace\limsup_{n\rightarrow\infty}\mathbb{P}\left\{ \left|\int f_{y}d\xi_{n}-\int f_{y}d\xi\right|>3\epsilon\right\} \leq\mathbb{P}\left\{ \int f_{y}\cdot h_{T}^{c}d\xi>\epsilon\right\} \\
 & +\limsup_{n\rightarrow\infty}\mathbb{P}\left\{ \int f_{y}\cdot h_{T}^{c}d\xi_{n}>\epsilon\right\} +\limsup_{n\rightarrow\infty}\mathbb{P}\left\{ \left|\int f_{y}\cdot h_{T}d\xi_{n}-\int f_{y}\cdot h_{T}d\xi\right|>\epsilon\right\}
\end{align*}
and note that as $T\rightarrow\infty$, by (\ref{eq:4}) and (\ref{eq:5}),
the first and second summands tend to $0$, and, from convergence
in probability, so does the third summand.

Therefore $\int f_{y}d\xi_{n}\rightarrow\int f_{y}d\xi$ in distribution
and thus $\lim_{n\rightarrow\infty}L_{n}\left[\left.f\,\right|\,\cdot\,\right]=L\left[\left.f\,\right|\,\cdot\,\right]$.

Note that if $\xi=0$ then, obviously, $L_{\xi}\left[\left.f\,\right|\,\cdot\,\right]=g\left(\cdot\right)=1$.
Assume henceforth that $\xi\neq0$. By the monotone convergence theorem,
for any $y\in\mathbb{R}$,
\begin{equation}
L_{\xi}\left[\left.f\,\right|\, y\right]=\lim_{T\rightarrow\infty}\mathbb{E}\left\{ \exp\left(-\left\langle \theta_{y}\left(f\cdot h_{T}\right),\xi\right\rangle \right)\right\} =\lim_{T\rightarrow\infty}g\left(y-\tau_{T}\right),\label{eq:3}
\end{equation}
where the second equality follows, with appropriate $\tau_{T}$, from
the fact that $f\cdot h_{T}\in C_{c}^{+}\left(\mathbb{R}\right)$
for large $T$ and since $L_{\xi}\left[\left.f\,\right|\,\cdot\,\right]$
is uniquely supported.

By Lemma \ref{lem:g monotone} and the fact that $\int f_{y}d\xi<\infty$
almost surely,
\begin{align*}
1 & >L_{\xi}\left[\left.f\cdot h_{T}\,\right|\, y\right]\geq L_{\xi}\left[\left.f\,\right|\, y\right]=\mathbb{E}\left\{ \exp\left(-\left\langle \theta_{y}f,\xi\right\rangle \right)\right\} \\
 & >\mathbb{E}\left\{ \exp\left(-\infty\cdot\mathbf{1}_{\left\{ \xi\left(\mathbb{R}\right)>0\right\} }\right)\right\} =\inf_{y\in\mathbb{R}}g\left(y\right)
\end{align*}
and thus, by Corollary \ref{cor:g limits}, $L_{\xi}\left[\left.f\,\right|\, y\right]=g\left(y-\tau\right)$,
with $\lim_{T\rightarrow\infty}\tau_{T}=\tau\in\mathbb{R}$, for all
$y\in\mathbb{R}$, which completes the proof.
\end{proof}

\section{\label{sec:Appendix_Gum}Appendix I: Gumbel distributions of different scales}
This short appendix is devoted to a proposition which relates Gumbel distribution functions of different scales. We suspect that it must be known to experts but
we have not been able to locate neither such a statement in the literature
nor a direct, analytic proof.
The proposition follows from the observation made in the proof of Lemma \ref{lem:150414-1}: by composing an SDPPP with
a Poisson process of exponential density one obtains a process which has two substantially different representation as an SDPPP.
\begin{prop}
\label{prop:gum}
For any $c_{2}>c_{1}>0$,
\begin{equation}
\label{eq:Gumconv}
{\rm {Gum}}\left(c_{1}\left(y\right)\right)=\int{\rm {Gum}}\left(c_{2}\left(y-z\right)\right)d\mu\left(z\right),
\end{equation}
where $\mu=\mu_{c_{1},c_{2}}$ is the law of $S-\tau$ for some $\tau\in\mathbb R$, with $S \triangleq \frac{1}{c_2}\log\left(
\left\langle e^{c_2 x},\xi^{c_1}\right\rangle \right)$ and $\xi^{c_1}\sim DPPP\left(e^{-c_1 x}dx,\delta_{0}\right)$.
\end{prop}
\begin{proof}
For any $c>0$ let $\xi^c\sim DPPP\left(e^{-cx}dx,\delta_{0}\right).$
Note that
\begin{align*}
\mathbb{E}\left\{ \left\langle e^{c_2 x},\left.\xi^{c_1}\right|_{\left(-\infty,0\right)}\right\rangle \right\}  & =\sum_{n=0}^{\infty}\mathbb{E}\left\{ \left\langle e^{c_2 x},\left.\xi^{c_1}\right|_{\left[-(n+1),-n\right)}\right\rangle \right\} \\
 & \leq\frac{1}{c_{1}}\sum_{n=0}^{\infty}e^{-c_2 n}\left(e^{c_{1}\left(n+1\right)}-e^{c_{1}n}\right)<\infty.
\end{align*}
Thus, since $\xi^{c_1}\left([0,\infty)\right)<\infty$ a.s., $\left\langle e^{c_2 x},\xi^{c_1}\right\rangle <\infty$
a.s.

Let
\[
\eta \sim SDPPP\left( e^{-c_1 x}dx, \xi^{c_2}, 0 \right).
\]
By the argument leading to
\eqref{eq:etaSDPPP}
in the proof of Lemma \ref{lem:150414-1}, we also have
\[
\eta \sim SDPPP\left( e^{-c_2 x}dx, \delta_0, S \right),
\]
where $S$ is defined in the statement of the proposition.

From the converse part of Theorem 9, $\eta$ satisfies (SUS) with both the functions (of $y$) appearing in the two sides of \eqref{eq:Gumconv}.
The converse part also gives $\tau$ explicitly, using  \eqref{eq:31}. This completes the proof.
\end{proof}

\section{\label{sec:Appendix}Appendix II: point processes}

Denote by $\mathcal{N}$ the space of positive, locally finite, counting
measures on $\mathbb{R}$. That is, Borel measures $\mu$ such that
$\mu\left(B\right)\in\mathbb{Z}_{+}=\left\{ 0,1,2,\ldots\right\} $
for any bounded Borel set $B$. Endow $\mathcal{N}$ with the vague
topology (cf. Appendix 7 of \cite{KAllenbergMeas}) and let $\mathcal{A}$
be the corresponding Borel $\sigma$-algebra on $\mathcal{N}$. Note
that with this choice of topology $\mathcal{N}$ is Polish (cf. A
7.7 of \cite{KAllenbergMeas}) and therefore $\left(\mathcal{N},\mathcal{A}\right)$
is a Borel space (a fact we shall need for some coupling arguments
in Section \ref{sec:mainpf}). Given a probability space $\left(\Omega,\mathcal{F},\mathbb{P}\right)$,
a point process (on $\mathbb{R}$) is a measurable mapping $\left(\Omega,\mathcal{F}\right)\rightarrow\left(\mathcal{N},\mathcal{A}\right)$.
(Since all the point processes we consider are on $\mathbb{R}$, henceforth
we shall not state the space on which they are defined)

The Laplace functional of a point process $\xi$ is the mapping $M^{+}\left(\mathbb{R}\right)\rightarrow\left[0,1\right]$
given by
\[
L_{\xi}\left[f\right]\triangleq\mathbb{E}\left\{ \exp\left(-\left\langle f,\xi\right\rangle \right)\right\} ,\,\,\,\, f\in M^{+}\left(\mathbb{R}\right).
\]
When discussing convergence of a sequence of point processes $\xi_{n}$,
$n\geq1$, we shall consider convergence in distribution with respect
to the vague topology which we denote by $\overset{d}{\rightarrow}$.
As Theorem 4.2 of \cite{KAllenbergMeas} states, convergence in distribution
of the point processes $\xi_{n}\overset{d}{\rightarrow}\xi$ is equivalent
to convergence in distribution of the random variables $\left\langle f,\xi_{n}\right\rangle \rightarrow\left\langle f,\xi\right\rangle $,
for any $f\in C_{c}^{+}\left(\mathbb{R}\right)$, and to convergence
of the Laplace functionals $L_{\xi_{n}}\left[f\right]\rightarrow L_{\xi}\left[f\right]$,
for any $f\in C_{c}^{+}\left(\mathbb{R}\right)$.

According to Lemma 4.5 of \cite{KAllenbergMeas}, relative compactness
of a sequence $\xi_{n}$, $n\geq1$, with respect to convergence in
distribution in the vague topology is equivalent to tightness of the
sequence which is equivalent to
\begin{equation}
\lim_{t\rightarrow\infty}\limsup_{n\rightarrow\infty}\mathbb{P}\left\{ \xi_{n}\left(B\right)>t\right\} =0,\,\,\,\,\mbox{for any bounded Borel set }B.\label{eq:30-1}
\end{equation}

In particular, it follows that if (\ref{eq:30-1}) holds and $L_{\xi_{n}}\left[f\right]\rightarrow\varphi\left[f\right]$,
for any $f\in C_{c}^{+}\left(\mathbb{R}\right)$, for some function
$\varphi:C_{c}^{+}\left(\mathbb{R}\right)\rightarrow\mathbb{R}$,
then there exists a point process $\xi$ such that $\varphi\left[f\right]=L_{\xi}\left[f\right]$
on $C_{c}^{+}\left(\mathbb{R}\right)$ and $\xi_{n}\overset{d}{\rightarrow}\xi$.

\subsection*{Acknowledgments}

The authors would like to thank Pascal Maillard for many helpful discussions.

\bibliographystyle{amsplain}
\bibliography{master}

\end{document}